\pdfoutput=1
\documentclass[a4paper]{amsart} 
\usepackage[utf8]{inputenc}
\usepackage[english]{babel}
\usepackage{amsmath}
\usepackage{fontenc}
\usepackage{amsfonts}
\usepackage{amssymb}
\usepackage{amsaddr}
\usepackage{comment}
\usepackage[makeroom]{cancel}
\usepackage{mathtools}
\usepackage{mathrsfs}
\usepackage{stmaryrd}
\usepackage{tikz-cd}

\usepackage{hyperref} 

\renewcommand{\eqref}[1]{(\ref{#1})}

\usepackage{graphicx}

\newtheorem{theorem}{Theorem}[section]

\newtheorem{lemma}[theorem]{Lemma}

\numberwithin{equation}{section}

\begin{document}
	
	\newcommand{\mMC}{\text{MC}}
	\newcommand{\md}{\text{d}}
	\newcommand{\comm}[2][red]{\textbf{\textcolor{#1}{#2}}} 
	\newcommand{\imp}[1]{\textbf{\emph{#1}}} 
	\newcommand{\eqd}{$\coloneqq$ } 
	\newcommand{\meqd}{\coloneqq} 
	\newcommand{\mad}{\text{ad} }
	\newcommand{\minc}[1]{\underset{#1}{\in}}
	\newcommand{\meqc}[1]{\underset{#1}{=}}
	\newcommand{\mtr}{\text{tr}}
	\newcommand{\abs}[1]{\lvert #1 \rvert}
	\newcommand{\ecomm}[2][green]{\textbf{\textcolor{#1}{#2}}} 
	\newcommand{\rnum}[1]{\uppercase\expandafter{\romannumeral #1\relax}}
	\newcommand{\mId}{\text{Id}}
	\newcommand{\mand}{\text{and}}
	\newcommand{\mst}{\text{s.t.}}
	\newcommand{\mmF}{\mathfrak{F}}
	\newcommand{\mmS}{\mathfrak{S}}
	\newcommand{\mmMC}{\mathfrak{MC}_\bullet}
	\newcommand{\mHom}{\text{Hom}}
	\newcommand{\mcC}{\mathcal{C}}
	\newcommand{\mcA}{\mathcal{A}}
	\newcommand{\mcD}{\mathcal{D}}
	\newcommand{\mcO}{\mathcal{O}}
	\newcommand{\mObj}{\text{Obj}}
	\newcommand{\mop}{\text{op}}
	\newcommand{\mSimp}{\text{Simp}}
	\newcommand{\mSet}{\text{Set}}
	\newcommand{\mcoker}{\text{coker}}
	\newcommand{\mker}{\text{ker}}
	\newcommand{\mIm}{\text{Im}}
	\newcommand{\mcone}{\text{cone}}
	\newcommand{\mmod}{\text{mod}}
	\newcommand{\mwith}{\text{with}}
	\newcommand{\mnew}{\text{new}}
	\newcommand{\mby}{\text{by}}
	\newcommand{\mmin}{\text{min}}
	\newcommand{\ubr}[1]{\underbrace{#1}}
	\newcommand{\nonu}{\nonumber}
	\newcommand{\mMF}{\mmF}
	\newcommand{\minitial}{\text{initial}}
	\newcommand{\mcurv}{\text{curv}}
	\newcommand{\mhK}{\mathbb{K}} 
	\newcommand{\mStub}{\text{Stub}}

	\title{A variation of the Goldman-Millson Theorem for filtered $L_\infty$ algebras}
	
	\author{Silvan Schwarz}
	\address{ETH Z\"urich}
	\email{silvan.schwarz@math.ethz.ch}

	\date{\today}
	
	\begin{abstract}
		In this paper, we extend the Goldman-Millson Theorem for $L_\infty$ algebras.\\
		We consider two $L_\infty$ algebras $L$ and $\tilde{L}$ endowed with descending, bounded above and complete filtrations compatible with the $L_\infty$ structures and ${U:L \rightarrow \tilde{L}}$ an $\infty$-morphism respecting the filtrations.\\
		We prove that in the setting of the linear part of $U$, say $\psi$, being a quasi-isomorphism on the r-1st page of the spectral sequences and\\
		${H^1 ((\mmF_{2^q} L)/(\mmF_{\mmin(2^{q+1},r)} L))=0}$ for every $q$ with $2^q < r$ and ${H^i((\mmF_1 \tilde{L}) / (\mmF_q \tilde{L}))=0}$ for $i=0,1$ and $q$ every power of 2 smaller than $r$ and $q=r$ this induces a weak homotopy equivalence of the simplicial sets $\mathfrak{MC}_\bullet (L)$ and $\mathfrak{MC}_\bullet (\tilde{L})$.
	\end{abstract}
	\maketitle
	\tableofcontents
	\section{Introduction}
	The Goldman-Millson Theorem was first introduced in \cite{GMTOriginal}, where it was formulated in the context of differential graded Lie algebras. Later, the foundational paper \cite{Getzler} by Getzler and the subsequent paper \cite{DolgushevLinfty} by Dolgushev and Rogers generalised the Goldman-Millson Theorem to $L_\infty$ algebras.\\
	They showed that for $L$ and $\tilde{L}$, two $L_\infty$ algebras equipped with descending, bounded above and complete filtrations compatible with the $L_\infty$ structures and ${U: L \rightarrow \tilde{L}}$ an $\infty$-morphism compatible with the filtrations, there is the following implication:
	If the linear term $\psi$ of $U$ gives a quasi-isomorphism ${\psi \vert_{\mmF_n L}: \mmF_n L \rightarrow \mmF_n \tilde{L}}$ for every $n \geq 1$, then $U$ induces a weak homotopy equivalence ${ \mmMC(U): \mmMC(L) \rightarrow \mmMC(\tilde{L}) }$.\\
	In this paper, we prove an analogue of this statement with a weaker quasi-isomorphism condition for the linear part of $U$, say $\psi$. Instead of demanding $\psi \vert_{\mmF_n L}$ to be a quasi-isomorphism for every $n$, we merely need $\psi$ (no restriction) to be a quasi-isomorphism on the r-1st page of the spectral sequences induced by the filtered complexes $(L, d_L)$ and $(\tilde{L}, d_{\tilde{L}})$ for some arbitrary $r$.\\
	At the cost of being able to weaken the quasi-isomorphism condition, we have to account for the additional requirement of vanishing co-homologies\\
	$H^1 ((\mmF_{2^q} L)/(\mmF_{\mmin(2^{q+1},r)} L))=0$ for every $q$ with $2^q < r$ and $H^i((\mmF_1 \tilde{L}) / (\mmF_q \tilde{L}))=0$ for $i=0,1$ and $q$ every power of 2 smaller than $r$ and $q=r$.\\
	The main result of this paper is:
	\begin{theorem}
		\label{Theorem W15F1}\hfill\\
		Let $L$ and $\tilde{L}$ be two $L_\infty$ algebras equipped with descending, bounded above and complete filtrations
		\begin{align*}
		\begin{aligned}
		L= \mmF_1 L \supset \mmF_2 L \supset \mmF_3 L \supset \ldots\\
		\tilde{L}= \mmF_1 \tilde{L} \supset \mmF_2 \tilde{L} \supset \mmF_3 \tilde{L} \supset \ldots
		\end{aligned}
		\end{align*}
		compatible with the $L_\infty$ structures.\\
		Let $U:L \rightarrow \tilde{L}$ be an $\infty$-morphism of $L_\infty$ algebras compatible with the filtrations that has its linear part, say $\psi$, being a quasi-isomorphism on the r-1st page\footnote{To avoid ambiguity: The quasi-isomorphism condition on the r-1st page has to be understood in the sense of being equivalent to $E_r (\mcone (\psi))=0$.} of the spectral sequences induced by the filtered complexes $(L,d_L)$ and $(\tilde{L}, d_{\tilde{L}})$ for some arbitrary $r$.\\
		Moreover, let us assume $H^1 ((\mmF_{2^q} L)/(\mmF_{\mmin(2^{q+1},r)} L))=0$ for every $q$ with $2^q < r$ and $H^i((\mmF_1 \tilde{L}) / (\mmF_q \tilde{L}))=0$ for $i=0,1$ and $q$ every power of 2 smaller than $r$ and $q=r$.\footnote{The requirements for $\tilde{L}$ can equivalently be formulated as $H^i ((\mmF_{2^q} \tilde{L})/(\mmF_{\mmin(2^{q+1},r)} \tilde{L}))=0$ for every $q$ with $2^q < r$ and $i=0,1$. $H^1 ((\mmF_{2^q} \tilde{L})/(\mmF_{\mmin(2^{q+1},r)} \tilde{L}))=0$ for every $q$ with $2^q < r$ and $H^0 ((\mmF_1 \tilde{L})/(\mmF_r \tilde{L}))=0$ would be another possible choice.}\\
		Then $U$ induces a weak homotopy equivalence of simplicial sets:
		\begin{align*}
		\mmMC(U):\mmMC(L) \rightarrow \mmMC(\tilde{L}).
		\end{align*}
	\end{theorem}
	This Theorem is a direct consequence of the following two:
	\begin{theorem}
		\label{Theorem W15F2}\hfill\\
		Let $U: L \rightarrow \tilde{L}$ be an $\infty$-morphism satisfying the conditions of Theorem \ref{Theorem W15F1}.\\
		Then $\mmMC(U)$ induces a bijection amongst the connected components:
		\begin{align*}
		\pi_0 ( \mmMC(L)) \stackrel{\cong}{\rightarrow} \pi_0 (\mmMC(\tilde{L})).
		\end{align*}
	\end{theorem}
	\begin{theorem}
		\label{Theorem W15F3}\hfill\\
		Let $U: L \rightarrow \tilde{L}$ be an $\infty$-morphism satisfying the conditions of Theorem \ref{Theorem W15F1}, except $H^i ((\mmF_1 \tilde{L}) /(\mmF_q \tilde{L} ))=0$, which no longer has to be imposed.\\
		Then for every Maurer-Cartan element $\tau \in \mMC(L)$, $\mmMC(U)$ induces group isomorphisms of higher $(n \geq 1)$ homotopy groups:
		\begin{align*}
		\pi_n (\mmMC(L), \tau) \stackrel{\cong}{\rightarrow} \pi_n (\mmMC(\tilde{L}), U_\star (\tau)).
		\end{align*}
	\end{theorem}\hfill\\
	The conditions on co-homologies and filtrations can severely be weakened in exchange for requiring $L$ to be Abelian and $U$ to raise the degree of filtration by $r-1$. A typical application for this additional variation would be the computation of automorphism groups as in e.g. \cite{FresseTW}.\newline
	There one studies the situation in which a dg Lie algebra acts on a deformation complex. In particular, one obtains an $\infty$-morphism from an Abelian $L_\infty$ algebra to an $L_\infty$ algebra by looking at the action of the Abelian part of the dg Lie algebra on the deformation complex (cf. \cite{FresseTW} Section 4.5).
	\begin{theorem}
		\label{Theorem ZA1}\hfill\\
		Let $L$ and $\tilde{L}$ be two $L_\infty$ algebras equipped with descending, bounded above and complete filtrations
		\begin{align*}
		\begin{aligned}
		L= \mmF_0 L \supset \mmF_1 L \supset \mmF_2 L \supset \ldots\\
		\tilde{L}= \mmF_1 \tilde{L} \supset \mmF_2 \tilde{L} \supset \mmF_3 \tilde{L} \supset \ldots
		\end{aligned}
		\end{align*}
		compatible with the $L_\infty$ structures. In addition, let $L$ be Abelian.\\
		Let $U:L \rightarrow \tilde{L}$ be an $\infty$-morphism of $L_\infty$ algebras compatible with the filtrations that has its linear part, say $\psi$, being a quasi-isomorphism on the r-1st page of the spectral sequences induced by the filtered complexes $(L,d_L)$ and $(\tilde{L}, d_{\tilde{L}})$ for some arbitrary fixed $r$. Moreover, let us assume $U$ to raise the degree of filtration by $r-1$\footnote{See Equation \eqref{Raise}} and $U_\star$ to be a finite sum\footnote{See Equation \eqref{W14B8}. This is equivalent to demanding that for every $a \in L$ there exists a $N \in \mathbb{N}$ such that for $U^\prime (\ubr{a, \ldots, a}_{\text{n-times}})=0$ for all $n \geq N$.}.\\
		Let also $H^1((\mmF_{2^q} \tilde{L})/(\mmF_{\mmin (2^{q+1},r)} \tilde{L}))=0$ hold for every $q$ with $2^q < r$.\\
		Then $\mmMC(U)$ induces a bijection amongst the connected components
		\begin{align*}
		\pi_0 ( \mmMC(L)) \stackrel{\cong}{\rightarrow} \pi_0 (\mmMC(\tilde{L})).
		\end{align*}
	\end{theorem}\hfill\\
	To facilitate the notation, we prove the theorems in the setting of shifted $L_\infty$ algebras. Nevertheless, due to the $\mathbb{Z}$-grading, they are still applicable to $L_\infty$ algebras.\\
	In Section \ref{Preliminaries}, we introduce some basic notation. Theorem \ref{Theorem W15F2} constitutes the main part of this paper, so we dedicate Sections \ref{Preparation} and \ref{Bijection} to its preparation and proof. We deal with higher homotopy groups in Section \ref{Higher}. A proof of Theorem \ref{Theorem ZA1} is given in Appendix \ref{Zusatzfrage}.
	\section{Acknowledgements}
	I would like to thank Thomas Willwacher for his great support during the supervision of the master project which led to this paper. Whenever I encountered difficulties, he made sure to take the time to answer my questions. I will remain indebted for having given me the chance for entering this interesting field of mathematics.\\
	The author has been partially supported by the ERC starting grant 678156 GRAPHCPX.
	\section{Preliminaries}
	\label{Preliminaries}
	Following \cite{DolgushevLinfty}, we work in the category of shifted $L_\infty$ algebras, denoted by $\mmS L_\infty$ (also see \cite{DolgushevEnhancement}). In short, an $\mmS L_\infty$ structure on $V$ can be thought as an $L_\infty$ structure on the suspension $V[-1]$. We will extensively borrow the notation from the two aforementioned papers.\\
	In the category of $\mmS L_\infty$ algebras, the Maurer-Cartan elements are defined by
	\begin{align}
	\label{W14B1}
	\mMC(L) \meqd \{ \alpha \in L^0 \vert \mcurv (\alpha) =0 \},
	\end{align}
	where $\mcurv$ is given by
	\begin{align}
	\label{W14B2}
	\mcurv (\alpha) \meqd \sum_{k=1}^{\infty} \frac{1}{k!} \{ \alpha, \ldots , \alpha  \}_k.
	\end{align}
	An advantage of working in $\mmS L_\infty$ is the simpler form of the Maurer-Cartan equation (compare Equation \eqref{W14B1} with e.g. Equation (1.7) from \cite{DolgushevLinfty}).\\
	In the theorems, we work with $\mmS L_\infty$ algebras endowed with descending, bounded above and complete filtrations compatible with the $\mmS L_\infty$ structures. This means for $L$ being of such type it must satisfy
	\begin{align}
	\label{W14B3}
	L = \mmF_1 L \supset \mmF_2 L \supset \mmF_3 L \supset \ldots,\\
	\label{W14B4}
	L = \lim\limits_{\leftarrow} L /\mmF_p L
	\end{align}
	and the degree of filtration must add under $\mmS L_\infty$ brackets.\\
	For the scope of this paper we denote the $\mmS L_\infty$ brackets by curly brackets, sometimes interchangeably writing $d_L (.)$ for the 1-bracket $\{.\}_1$. An $\mmS L_\infty$ algebra is said to be Abelian if $\{ \ldots \}_{n} = 0$ for $n \geq 2$.\\
	Furthermore, we deal with $\infty$-morphisms compatible with the filtrations.\\
	Since an $\infty$-morphism $U: L \rightarrow \tilde{L}$ is uniquely determined by its composition with the projection
	\begin{align}
	\label{W14B5}
	U^\prime \meqd \pi_{\tilde{L}} \circ U : S^{+} (L) \rightarrow \tilde{L},
	\end{align}
	it is sufficient to work with $U^\prime$ exclusively.\\
	In this manner, compatibility means
	\begin{align}
	\label{W14B6}
	U^\prime ( \mmF_{i_1}L \otimes \mmF_{i_2}L \otimes \ldots \otimes \mmF_{i_p}L) \subset \mmF_{i_1 + i_2 + \ldots + i_p} \tilde{L},
	\end{align}
	i.e. the degree of filtration adds in the argument of $U^\prime$.\\
	We say $U$ raises the degree of filtration by $r-1$, if 
	\begin{align}
	\label{Raise}
	U^\prime ( \mmF_{i_1}L \otimes \mmF_{i_2}L \otimes \ldots \otimes \mmF_{i_p}L) \subset \mmF_{i_1 + i_2 + \ldots + i_p+ (r-1)} \tilde{L}
	\end{align}
	holds.\\
	We denote the restriction of $U^\prime$ to $L$ by
	\begin{align}
	\label{W14B7}
	\psi \meqd U^\prime \vert_{L} : L \rightarrow \tilde{L}
	\end{align}
	and due to the multilinearity of $U$ (as $\infty$-morphism) it is reasonable to call this map the linear term of $U$.\\
	Eventually, we define the map
	\begin{align}
	\label{W14B8}
	\left\{
	\begin{aligned}
	U_\star :L^{0} &\rightarrow \tilde{L}^{0}\\[-10pt]
	\alpha &\mapsto U_\star (\alpha) \meqd \sum_{m=1}^{\infty} \frac{1}{m!} U^\prime (\ubr{\alpha, \ldots, \alpha}_{\text{m-times}}).
	\end{aligned}
	\right.
	\end{align}
	On several occasions, we will come across the notions of twisted $\mmS L_\infty$ algebras and twisted morphisms.\\
	Let $L$ be a filtered $\mmS L_\infty$ algebra with the $\mmS L_\infty$ brackets denoted by $\{ \ldots \}$ and let $\tau \in \mMC(L)$ be a Maurer-Cartan element.\\
	Then, there is a filtered $\mmS L_\infty$ algebra $L^\tau$, given by the underlying vector space $L$ and the brackets
	\begin{align}
	\label{W15C1}
	\{ v_1, v_2, \ldots, v_n  \}_n^\tau \meqd \sum_{k=0}^{\infty} \frac{1}{k!} \{ \ubr{\tau, \ldots, \tau}_{\text{k-times}}, v_1, v_2 , \ldots, v_n  \}_{k+n}.
	\end{align} 
	The filtration on $L^\tau$ is the one inherited from $L$ as they both have the same underlying filtered and graded vector space.\\
	In an analogous manner, an $\infty$-morphism ${U:L \rightarrow \tilde{L}}$ compatible with the filtrations induces an $\infty$-morphism 
	\begin{align}
	\label{W15C2}
	U^{\tau} :L^\tau  \rightarrow \tilde{L}^{U_\star (\tau)}
	\end{align}
	of the corresponding twisted $\mmS L_\infty$ algebras which is also compatible with the filtrations.\\
	By the notion of degree and from $d_L$ being a differential, it is clear that $(L, d_L)$ and $(\tilde{L},d_{\tilde{L}})$ can also be considered as filtered complexes. In this sense, ${\psi : L \rightarrow \tilde{L}}$ can be seen as a chain map.\\
	Furthermore,
	\begin{align*}
	\mcone (\psi) \meqd  (L \oplus \tilde{L}[-1], d),
	\end{align*}
	with the differential $d$ given by
	\begin{align*}
	\begin{aligned}
	d : L \oplus \tilde{L} [-1] &\rightarrow L \oplus \tilde{L} [-1]\\
	(\alpha, \beta) &\mapsto d_L (\alpha) + \left( -d_{\tilde{L}} (\beta) + \psi (\alpha) \right) 
	\end{aligned}
	\end{align*}
	defines a chain complex on $L\oplus \tilde{L}[-1]$.\\
	Following the same line as \cite{MasterSpectral}, we consider the spectral sequences on $(L, d_L)$, $(\tilde{L}, d_{\tilde{L}})$ and $(L \oplus \tilde{L}[-1], d)$, whereby from now on we denote the latter by $\mcone(\psi)$.\\
	One can show that $\psi$ being a quasi-isomorphism on the r-1st page of the spectral sequences is equivalent to $E_r (\mcone(\psi) ) = 0$.\\
	When neglecting degrees in our notation (but not neglecting filtration degrees), we find the $n$th page of the spectral sequence of $\mcone (\psi)$ to be
	\begin{align*}
	\begin{aligned}
	&E_n^p (\mcone(\psi))= \{  u \in \mmF_p (L \oplus \tilde{L} ) \vert d (u) \in \mmF_{p+n} (L \oplus \tilde{L})    \}  \\
	&/ \big\{ \mmF_{p+1} (L \oplus \tilde{L}) + \{d (w) \vert 	w \in \mmF_{p-n+1} (L \oplus \tilde{L}) ~\mand~d (w) \in \mmF_{p} (L \oplus \tilde{L}) \} \big\}.
	\end{aligned}
	\end{align*}
	Hence, $E_r (\mcone (\psi))=0$ translates to
	\begin{align*}
	\begin{aligned}
	&\big\{ \{ (a,b) \in \mmF_p (L \oplus \tilde{L}) \vert d(a,b) \in \mmF_{p+r} (L \oplus \tilde{L})  \}\\
	&- \{ d(x,y) \vert (x,y) \in \mmF_{p-r+1} (L \oplus \tilde{L}) ~\mand~d(x,y) \in \mmF_p ( L \oplus \tilde{L})\} \big\} \subset \mmF_{p+1} (L \oplus \tilde{L})~\forall p.
	\end{aligned}
	\end{align*}
	Unravelling the definitions of both $\mcone(\psi)$ and its differential results in
	\begin{align*}
	\text{"$\psi$ is a quasi-isomorphism on the r-1st spectral page"}
	\end{align*}
	being equivalent to the following statement to be valid for all $p$:
	\begin{align}
	\nonu &\text{If}~ a \in \mmF_p L~\mand~b \in \mmF_p \tilde{L}~\text{satisfy}:\\
	\label{W14B14a}
	& \hspace{2cm} \left\{
	\begin{aligned}
	d_L (a ) \in \mmF_{p+r} L\\
	\psi (a) - d_{\tilde{L}} (b) \in \mmF_{p+r} \tilde{L}.
	\end{aligned}
	\right.\\
	\nonu &\text{Then}~ \exists x \in \mmF_{p-r+1} L,~\exists y \in \mmF_{p-r+1} \tilde{L}~\text{such that}:\\
	\label{W14B14}
	& \hspace{2cm} \left\{
	\begin{aligned}
	d_L(x) \in \mmF_p L\\
	\psi(x)- d_{\tilde{L}} (y) \in \mmF_p \tilde{L}\\
	a- d_L (x) \in \mmF_{p+1} L\\
	b - \psi (x) + d_{\tilde{L}} (y) \in \mmF_{p+1} \tilde{L}.
	\end{aligned}	
	\right.
	\end{align}
	We use the convention of co-homology. That is by $H^k$ we denote the $k$th co-homology group and the differentials are supposed to raise the degrees by one.\\
	We continue by introducing the Maurer-Cartan functor
	\begin{align}
	\mmMC : \{\text{complete $\mmS L_\infty$-algebras}\} \rightarrow \mathbf{sSet}.
	\end{align}
	$\mmMC(L)$ denotes the simplicial set originating from \vspace{2pt}
	\begin{align}
	\label{W14B15}
	\mathfrak{MC}_n (L) \meqd \mMC(L \hat{\otimes } \Omega_n),~\text{ with }~	L \hat{\otimes } \Omega_n \meqd \lim_{\substack{\leftarrow\\
			p}}~ (L \otimes \Omega_n) / (\mmF_p L \otimes \Omega_n),
	\end{align}
	where elements in $\mathfrak{MC}_n (L)$ will be called n-cells. In particular, 1-cells of the form ${\beta = \beta_0 (t) + dt ~\beta_1}$, i.e. no time dependence in $\beta_1$, are said to be rectified.\\
	An $\infty$-morphism ${U:L \rightarrow \tilde{L}}$ induces a collection of $\infty$-morphisms via
	\begin{align}
	\label{W15G1}
	\left\{
	\begin{aligned}
	U^{(n)}:~L \hat{\otimes } \Omega_n &\rightarrow \tilde{L} \hat{\otimes } \Omega_n\\[-3pt]
	(v_1 \hat{\otimes} \omega_1, v_2 \hat{\otimes} \omega_2, \ldots, v_q \hat{\otimes} \omega_q) &\mapsto \epsilon U(v_1, \ldots, v_q) \hat{\otimes } \omega_1 \omega_2 \ldots \omega_q,
	\end{aligned}
	\right.	
	\end{align}
	where $v_i \in L$, $\omega_i \in \Omega_n$ and $\epsilon$ denotes the Koszul sign. The $\mmS L_\infty$ structures on $L \hat{\otimes } \Omega_n$ and $\tilde{L} \hat{\otimes } \Omega_n$ emerge from the $\mmS L_\infty$ structures on $L$ and $\tilde{L}$ in similar manners. Remarkably enough, the differential on this completion consists of the usual $d_L$ and an additional $d_{\Omega_n}$.\\
	With $U_\star^{(n)}$ defined as in Equation \eqref{W14B8}, we have
	\begin{align}
	\mathfrak{MC}_n (U) \meqd U_\star^{(n)} : \mathfrak{MC}_n (L) \rightarrow \mathfrak{MC}_n (\tilde{L}),
	\end{align}
	which forms a morphism of simplicial sets
	\begin{align}
	\label{W15G2}
	\mmMC(U):\mmMC(L) \rightarrow \mmMC(\tilde{L}).
	\end{align}
	According to Proposition 4.1 from \cite{DolgushevEnhancement}, $\mmMC(L)$ is a Kan complex. Therefore, it is sensible to introduce the concept of homotopy groups.\\
	Let $\tau \in \mMC(L)$ be a Maurer-Cartan element.\\
	The elements of the nth homotopy group $\pi_n (\mmMC(L),\tau)$ are given by the equivalence classes $ [ a ] $, where a representative $a$ is an $n$-cell in $\mmMC(L)$, satisfying $\partial_i a = \tau$ for all $i$, when $\tau$ is considered as a degenerate $(n-1)$-simplex. The notion of equivalence relation is such that $a \sim b$ iff there is a $(n+1)$-simplex $\omega$ fulfilling $\partial \omega \meqd (\partial_0 \omega, \ldots, \partial_{n+1} \omega) = (\tau, \ldots, \tau, a,b)$.\\
	The group structure emerges by defining $[a] \cdot [b] =: [\partial_n s]$, where $s$ is any $(n+1)$-simplex satisfying $\partial s = (\tau, \ldots, \tau, a, -, b)$. This definition makes sense, as from $\mmMC(L)$ being Kan, such $(n+1)$-simplex exists.\\
	From this definition it is clear that $\mMC(L)/\sim$, i.e. the quotient space of Maurer-Cartan elements with gauge-equivalence induced by 1-cells in $\mmMC(L)$, and the connected components $\pi_0 (\mmMC(L))$ are the same.\\
	Eventually, $\mmMC(U): \mmMC(L) \rightarrow \mmMC(\tilde{L})$ being a weak homotopy equivalence means:
	\begin{enumerate}
		\item $\mmMC(U)$ induces an isomorphism of connected components
		\begin{align*}
		\pi_0 (\mmMC(L)) \stackrel{\cong}{\rightarrow} \pi_0 (\mmMC(\tilde{L}))
		\end{align*}
		in $\mathbf{Set}$.
		\item For all Maurer-Cartan elements $\tau\in \mMC(L)$ and for all $n \geq 1$, $\mmMC(U)$ induces an isomorphism
		\begin{align*}
		\pi_n (\mmMC(L), \tau) \stackrel{\cong}{\rightarrow} \pi_n (\mmMC(\tilde{L}), U_\star (\tau))
		\end{align*}
		in $\mathbf{Grp}$.
	\end{enumerate}
	\section{Preparation for Bijection on the Level of $\pi_0$}
	\label{Preparation}
	Before stating the main theorem of this paper, we have to give four technical lemmata. The first one describes an interplay between 1-cells in $\mmMC(L)$ and initial value problems (IVP). The second one allows, under some conditions, for a 1-cell in $\mmMC(L)$ connecting two Maurer-Cartan elements to find another 1-cell in $\mmMC(L)$ connecting the same Maurer-Cartan elements and have its $dt$ term adjusted by an exact term. The third describes how every Maurer-Cartan element is found to be gauge-equivalent to a Maurer-Cartan element having at least a specific degree of filtration, subject to certain pre-conditions. Eventually, the fourth makes a useful statement for degree 0 elements which have differentials of higher filtration degrees.\\
	\begin{lemma}[1-Cell in $\mmMC(L)$ iff solution to IVP]
		\label{Lemma W14B1}\hfill\\	
		Let $L$ be an $\mmS L_\infty$ algebra endowed with a descending, bounded above\footnote{In case of $L$ being Abelian, the Lemma also holds for a filtration of the form ${L= \mmF_0 L \supset \mmF_1 L \supset \mmF_2 L \supset \ldots}$ instead.} and complete filtration
		\begin{align*}
			L= \mmF_1 L \supset \mmF_2 L \supset \mmF_3 L \supset \ldots
		\end{align*}
		compatible with the $\mmS L_\infty$ structure.\\
		Let us assume 
		\begin{align*}
		\beta = \beta_0 (t) + dt~\beta_1(t)
		\end{align*}
		for some $\beta_0 (t), \beta_1 (t)$ in $\mmF_1 L \hat{\otimes } \mhK[t]$.\\
		Let $m_0 \in \mMC(L)$ be a Maurer-Cartan element.\\
		Then $\beta= \beta_0 (t) + dt~\beta_1 (t)$ is a 1-cell in $\mmMC(L)$ with starting point $m_0$	iff $\beta_0 (t)$ is the (unique) solution to the IVP
		\begin{align}
		\label{W14B16}
		\left\{
		\begin{aligned}
		\frac{\partial \beta_0 (t)}{\partial t} &= d_L^{\beta_0 (t)} (\beta_1 (t))\\
		\beta_0 (0) &= m_0.
		\end{aligned}
		\right.
		\end{align}
	\end{lemma}
	\begin{proof}[\textbf{Proof of Lemma \ref{Lemma W14B1}}]
		\hfill\\
		In very much the same vein as in Section 6.2.3 from \cite{Nicoud} for $\beta$ to be a 1-cell in $\mmMC(L)$, we need $\mcurv (\beta)=0$ to hold.\\
		Exploiting $dt^2=0$ and using the definition of $\mmS L_\infty$ brackets, we rewrite this as
		\begin{align*}
		\begin{aligned}
		\mcurv (\beta) &= (d_L + d_{\Omega_1}) \big(\beta_0 (t)+ dt \beta_1 (t)\big) + \sum_{m=2}^{\infty} \frac{1}{m!} \{ \beta, \ldots, \beta \}_m \\
		&=\ubr{d_L (\beta_0 (t)) + \sum_{m=2}^{\infty} \frac{1}{m!} \{\beta_0 (t), \ldots, \beta_0 (t)\}}_{ = \mcurv (\beta_0 (t)) } - \frac{\partial \beta_0 (t)}{\partial t}  dt\\
		&+d_L (\beta_1 (t)) dt~+ \sum_{m=1}^{\infty} \frac{1}{m!} \{ \beta_0 (t), \ldots, \beta_0 (t), \beta_1 (t)  \}_{m+1} dt~ =0. 
		\end{aligned}
		\end{align*}
		So, using the definition of the twisted differential (cf. Equation \eqref{W15C1}) $\beta$ being a 1-cell in $\mmMC(L)$ is equivalent to
		\begin{align}
		\label{W14B19}
		\frac{\partial \beta_0(t) }{\partial t} = d_L^{\beta_0 (t)} ( \beta_1 (t))
		\end{align}
		and
		\begin{align}
		\label{W14B20}
		\mcurv ( \beta_0 (t)) =0~\text{for all $t$}.
		\end{align}
		The first implication of the statement follows directly from Equation \eqref{W14B19}.\\
		For the second implication, Equation \eqref{W14B19} is obviously satisfied, hence it remains to confirm Equation \eqref{W14B20}.\\
		We have
		\begin{align}
		\label{W14B21}
		\left\{
		\begin{aligned}
		&\frac{\partial (\mcurv (\beta_0 (t)))}{\partial t} = \frac{\partial}{\partial t} \sum_{k=1}^{\infty} \frac{1}{k!} \{ \beta_0 (t) , \ldots, \beta_0 (t)\}_k \\
		&=\sum_{k=0}^{\infty} \frac{1}{k!} \{ \beta_0 (t), \ldots, \beta_0 (t), \frac{\partial}{\partial t} \beta_0 (t) \}_{k+1} = d_L^{\beta_0 (t)} \Big( \frac{\partial \beta_0 (t)}{\partial t} \Big)\\
		&\hspace{-4pt}\meqc{\eqref{W14B16}}d_L^{\beta_0 (t)} \circ d_L^{\beta_0 (t)} (\beta_1 (t)) \meqc{\text{(2.19) from \cite{DolgushevEnhancement}}} - \{\mcurv (\beta_0 (t)), \beta_1 (t)\}_2^{\beta_0 (t)}
		\end{aligned}
		\right.
		\end{align}
		and $\beta_0 (0)=m_0$.\\
		As $m_0 \in \mMC(L)$ yields $\mcurv(m_0) =0$, we find $\mcurv (\beta_0 (0))=0$ and together with Equation \eqref{W14B21} this ensures Equation \eqref{W14B20} to hold.
	\end{proof}
	\begin{lemma}
		\label{Lemma W14B2}\hfill\\
		Let $L$ be an $\mmS L_\infty$ algebra endowed with a descending, bounded above and complete filtration
		\begin{align*}
		L= \mmF_1 L \supset \mmF_2 L \supset \mmF_3 L \supset \ldots
		\end{align*}
		compatible with the $\mmS L_\infty$ structure.\\
		Let 
		\begin{align*}
		\beta = \beta_0 (t) + dt~\beta_1 (t)
		\end{align*}
		be a 1-cell in $\mmMC (L)$ connecting the two Maurer-Cartan elements $m_0, m_1 \in \mmF_q L$.\\
		Let $\beta_1 (t)$ satisfy $\beta_1 (t) \in \mmF_p L \hat{\otimes} \mathbb{K}[t]$.\\
		Let us assume both $y \in \mmF_{p-l} L$ and $\beta_1 (t) + d_L (y) \in \mmF_{p+1} L \hat{\otimes} \mathbb{K}[t]$ to hold.\\
		If
		\begin{align}
		\label{W14B22ab}
		q \geq l+1 ~\mand ~ p \geq l+1,
		\end{align}
		then there exists a rectified 1-cell
		\begin{align*}
		\gamma = \gamma_0 (t) + dt~\gamma_1
		\end{align*}
		in $\mmMC(L)$ connecting $m_0$ and $m_1$ and having $\gamma_1 \in \mmF_{p+1} L$.
	\end{lemma}
	\begin{proof}[\textbf{Proof of Lemma \ref{Lemma W14B2}}]
		\hfill\\	
		We set
		\begin{align}
		\label{W14B22}
		\mu \meqd \beta + d_L^{\beta} (y~dt).
		\end{align}
		It is clear that when denoting 
		\begin{align}
		\label{W14B23}
		U \meqd d_L^{\beta} (y~dt),
		\end{align}
		Equation \eqref{W14B22} takes the form
		\begin{align}
		\label{W14B24}
		\mu = \beta + U.
		\end{align}
		For $\mu$ being a 1-cell in $\mmMC(L)$, it must satisfy $\mcurv (\mu) =0$.\\
		A short calculation proves
		\begin{align*}
		\begin{aligned}
		\mcurv (\mu) &\meqc{\eqref{W14B24}} \mcurv(\beta + U)  \\
		&\hspace{-17pt}\meqc{\text{(2.20) from \cite{DolgushevEnhancement}}}\ubr{\mcurv (\beta)}_{\meqc{\beta \in \mmMC(L)} \hspace{-12pt} 0} + d_L^{\beta} (U) + \ubr{\sum_{m=2}^{\infty} \frac{1}{m!} \{ U, \ldots, U \}_m^{\beta}}_{= 0} \\
		&\hspace{4pt}=d_L^{\beta} (U) \meqc{\eqref{W14B23}} 	d_L^{\beta} (d_L^{\beta} (y~dt)) \meqc{\text{(2.19) from \cite{DolgushevLinfty}}} 	- \{ \ubr{\mcurv (\beta)}_{\meqc{\beta \in \mmMC(L)} 0} , y~dt  \}_2^{\beta} =0,
		\end{aligned}
		\end{align*}
		where we used the fact that every summand consists of at least 2 $U$ terms, but $U$ involves $dt$ and thus ($dt^2 = 0$) the sum vanishes.\\
		So we assured ourselves that $\mu$ is a 1-cell in $\mmMC(L)$, indeed.\\
		Unravelling the definition of the twisted differential and once more using the fact that $dt^2=0$, we can rewrite Equation \eqref{W14B22} as
		\begin{align}
		\label{W14B26}
		\mu = \beta + \sum_{k=0}^{\infty} \frac{1}{k!} \{  \beta_0 (t), \ldots , \beta_0 (t), y \}_{k+1} ~dt.
		\end{align}
		It is obvious, that $\mu \vert_{t=0} = \beta_0 (0)=m_0$ and $\mu \vert_{t=1} =\beta_0 (1) =m_1$, thus $\mu$ also connects $m_0$ with $m_1$.\\
		By means of Lemma \ref{Lemma W14B1}, $\beta$ being a 1-cell in $\mmMC(L)$ with starting point $m_0$ implies $\beta_0(t)$ to solve the IVP
		\begin{align*}
		\begin{aligned}
		\frac{\partial \beta_0 (t)}{\partial t} &= d_L^{\beta_0 (t)} (\beta_1 (t))\\
		\beta_0 (0) &= m_0.
		\end{aligned}
		\end{align*}
		Integrating and using the assumptions about the degrees of filtration yields
		\begin{align*}
		\beta_0 (t) = \ubr{m_0}_{\in \mmF_q L} + \int_{0}^{t} dt_1~\ubr{\sum_{k=0}^{\infty} \frac{1}{k!} \{ \beta_0(t_1), \ldots, \beta_0 (t_1), \beta_1 (t_1)  \}_{k+1}}_{\in \mmF_p L},
		\end{align*}
		i.e.
		\begin{align}
		\label{W14B29}
		\beta_0 (t) \in \mmF_{\text{min}(q,p)} L \hat{\otimes } \mhK[t].
		\end{align}
		Explicitly expanding Equation \eqref{W14B26} and grouping terms leads to
		\begin{align}
		\label{W14B30}
		\mu = \beta_0 (t) + \bigg( \ubr{\beta_1 (t) + d_L (y) }_{\in \mmF_{p+1} L \hat{\otimes } \mhK[t]} + \ubr{ \{ \beta_0 (t) , y \}_2 }_{\in \mmF_s L \hat{\otimes } \mhK[t]} + \mcO (\mmF_{s+1 } L \hat{\otimes } \mhK[t]) \bigg) dt
		\end{align}
		for some to be determined $s$.\\
		If we can show that $s \geq p+1$, then we are done.\\
		But we have
		\begin{align*}
		\{ \beta_0 (t), y \}_2 \minc{\eqref{W14B29}} \mmF_{(p-l)+ (\mmin(q,p))} L \hat{\otimes } \mhK[t] \underset{\eqref{W14B22ab}}{\subset} \mmF_{p+1} L \hat{\otimes } \mhK[t],
		\end{align*}
		hence $ \mu = \mu_0 (t) + dt~\mu_1 (t)$ is a 1-cell in $\mmMC(L)$ connecting $m_0$ with $m_1$ and having $\mu_1 (t) \in \mmF_{p+1} L \hat{\otimes } \mhK[t]$.\\
		By means of Lemma B2 from \cite{DolgushevLinfty}, there exists a rectified 1-cell in $\mmMC (L)$, which connects $m_0$ and $m_1$ and also has the coefficient of its $dt$ part in $\mmF_{p+1} L$.\\
	\end{proof}	
	\begin{lemma}
		\label{Lemma W19A1}\hfill\\
		Let $L$ be an $\mmS L_\infty$ algebra endowed with a descending, bounded above and complete filtration
		\begin{align*}
		L= \mmF_1 L \supset \mmF_2 L \supset \mmF_3 L \supset \ldots
		\end{align*}
		compatible with the $\mmS L_\infty$ structure.\\
		Let us assume $H^0 ( ( \mmF_{2^q} L)/(\mmF_{\mmin(r,2^{q+1})} L))=0$ for every q with $2^q < r$.\\
		Then for every Maurer-Cartan element $a \in \mMC(L)$ there exists a Maurer-Cartan element $a_r \in \mMC(L)$ which is gauge equivalent to the initial one $a_r \sim a$ for the gauge equivalence induced by 1-cells in $\mmMC(L)$ and satisfies $a_r \in \mmF_r L$.\\
		The same statement also holds with $H^i ((\mmF_1 L)/(\mmF_q L))=0$ for $i=-1,0$ and $q$ every power of 2 smaller than $r$ and $q=r$, instead.\footnote{In the first version the proof was given for the latter case only.}
	\end{lemma}
	\begin{proof}[\textbf{Proof of Lemma \ref{Lemma W19A1}}]
		\hfill\\
		\underline{From $a_u \in \mmF_{2^u} L$ to $a_{u+1} \in \mmF_{2^{u+1}} L$ for $2^{u+1} <r$:}\hfill\\
		Because of the Maurer-Cartan equation, it is clear that $a_u \in \mmF_{2^u} L$ and $a_u \in \mMC(L)$ lead to
		\begin{align*}
		d_L (a_u) = \ubr{- \sum_{m=2}^{\infty} \frac{1}{m!} \{ a_u, \ldots , a_u \}_m }_{\in \mmF_{2 \cdot 2^u} L = \mmF_{2^{u+1}} L} \in \mmF_{2^{u+1}} L .
		\end{align*}
		Due to the assumption, $H^0((\mmF_{2^u} L) /( \mmF_{2^{u+1}} L ))=0$ holds and we find
		\begin{align}
		\label{W14Z1}
		\left\{
		\begin{aligned}
		\exists q_{u+1} &\in \mmF_{2^u} L~\mst\\
		a_u- d_{L} (q_{u+1}) &\in \mmF_{2^{u+1}} L.
		\end{aligned}
		\right.
		\end{align}
		We construct a rectified 1-cell $\sigma^{(u+1)} = \sigma_0^{(u+1)} (t) + dt~\sigma_1^{(u+1)}$ in $\mmMC(L)$ by setting its starting point to $a_u$ and demanding
		\begin{align}
		\label{W14Z2}
		\sigma_1^{(u+1)} \meqd -q_{u+1}.
		\end{align}
		By means of Lemma \ref{Lemma W14B1}, this describes a 1-cell in $\mmMC(L)$ if $\sigma_0^{(u+1)}(t)$ is said to be the solution to the IVP
		\begin{align}
		\label{W14Z3}
		\left\{
		\begin{aligned}
		\frac{\partial \sigma_0^{(u+1)} (t)}{\partial t} &= d_L^{\sigma_0^{(u+1)} (t)} ( \sigma_1^{(u+1)})\\
		\sigma_0^{(u+1)} (0) &= a_u.
		\end{aligned}
		\right.
		\end{align}
		We set $a_{u+1}$ to be the endpoint of this so-constructed 1-cell, i.e.
		\begin{align}
		\label{W14Z4}
		a_{u+1} \meqd \sigma^{(u+1)} \vert_{t=1} = \sigma_0^{(u+1)} (1).
		\end{align}
		Exploiting Equation \eqref{W14Z1} and integrating Equation \eqref{W14Z3} yields
		\begin{align}
		\label{W14Z4b}
		\begin{aligned}
		&\sigma_0^{(u+1)} (t) = \ubr{a_{u}}_{\in \mmF_{2^u} L} + \int_{0}^{t} dt_1 \Big( \ubr{d_{L} ( - q_{u+1} )}_{\minc{\eqref{W14Z1}}\mmF_{2^u} L}\\[-6pt]
		&+\sum_{m=1}^{\infty} \frac{1}{m!} \{ \sigma_0^{(u+1)}(t_1), \ldots, \sigma_0^{(u+1)} (t_1), \ubr{- q_{u+1} }_{\minc{\eqref{W14Z1}} \mmF_{2^u} L}   \}_{m+1} \Big) \in \mmF_{2^u} L \hat{\otimes } \mhK[t],
		\end{aligned}
		\end{align}
		i.e.
		\begin{align}
		\label{W14Z4c}
		\sigma_0^{(u+1)} ( t) \in \mmF_{2^u} L \hat{\otimes } \mhK[t].
		\end{align}
		This, in turn, can be plugged into the integral up to $t=1$, leading to
		\begin{align}
		\left\{
		\begin{aligned}	
		a_{u+1} &\meqd \sigma_0^{(u+1)} (1) = a_u + \int_{0}^{1} dt \Big( d_L (-q_{u+1})\\
		&+\ubr{\sum_{m=1}^\infty \frac{1}{m!} \{ \ubr{\sigma_0^{(u+1)} (t)}_{\minc{\eqref{W14Z4c}} \mmF_{2^u} L \hat{\otimes } \mhK[t]}, \ldots , \sigma_0^{(u+1)} (t), \ubr{- q_{u+1} }_{\minc{\eqref{W14Z1}} \mmF_{2^u} L}  \}_{m+1}}_{\in \mmF_{2^{u+1}} L \hat{\otimes } \mhK[t]} \Big)\\
		&=\ubr{a_u - d_L (q_{u+1})}_{\minc{\eqref{W14Z1}} \mmF_{2^{u+1} L}} + \mcO ( \mmF_{2^{u+1}} L) \in \mmF_{2^{u+1}} L.
		\end{aligned}	
		\right.	
		\end{align}
		We continue with this procedure until we eventually arrive at the situation in which we find a Maurer-Cartan element $a_s \sim \ldots \sim a$ which has $a_s \in \mmF_{2^s} L$ for some $s$ with $2^s \leq r < 2^{s+1}$.\\
		\underline{From $a_s \in \mmF_{2^s} L$ to $a_r \in \mmF_r L$ for $2^s < r < 2^{s+1}$:}\hfill\\
		With some minor adjustments the same steps can be used to construct an $a_r \in \mmF_r L$ having $a_r \sim a_s$ from such an $a_s$.\\
		Because $a_{s} \in \mmF_{2^s} L$ is a Maurer-Cartan element,
		\begin{align*}
		d_{L} (a_{s}) = \ubr{- \sum_{m=2}^{\infty} \frac{1}{m!} \{ a_{s}, \ldots , a_{s}  \}_m }_{\in \mmF_{2 \cdot 2^s} L} \in \mmF_{2^{s+1}} L \underset{2^{s+1} > r}{\subset} \mmF_r L
		\end{align*}
		holds.\\
		Furthermore, $H^0((\mmF_{2^s} L) /(\mmF_r L)) =0$ then implies
		\begin{align}
		\label{W14Z6}
		\left\{
		\begin{aligned}
		\exists q_{s+1} &\in \mmF_{2^s} L~\mst\\
		a_{s} - d_{L} ( q_{s+1}) &\in \mmF_r L.
		\end{aligned}
		\right.	
		\end{align}
		Constructing a rectified 1-cell $\sigma^{(s+1)}= \sigma_0^{(s+1)}(t) +dt~ \sigma_1^{(s+1)}$ in $\mmMC(L)$ in much the same way as before by setting its starting point to $a_{s}$ and demanding
		\begin{align}
		\label{W14Z7}
		\sigma_1^{(s+1)} \meqd - q_{s+1},
		\end{align}
		leads to (c.f. Equation \eqref{W14Z4c})
		\begin{align}
		\label{W14Z8}
		\sigma_0^{(s+1)} (t) \in \mmF_{2^s} L \hat{\otimes } \mhK[t].
		\end{align}
		Moreover, we set $a_{r}$ to be the endpoint of this 1-cell, i.e.
		\begin{align}
		\label{W14Z9}
		a_{r} \meqd \sigma^{(s+1)} \vert_{t=1} = \sigma_0^{(s+1)} (1),
		\end{align}
		and find its degree of filtration to be
		\begin{align}
		\label{W14Z10}
		\left\{
		\begin{aligned}
		a_{r} &\meqd \sigma_0^{(s+1)} (1) = a_{s} + \int_{0}^{1} dt \Big( \hspace{-3pt}-d_{L} ( q_{s+1} )\\
		&+\ubr{\sum_{m=1}^{\infty} \frac{1}{m!} \{ \ubr{\sigma_0^{(s+1)} (t)}_{\minc{\eqref{W14Z8}} \mmF_{2^s} L \hat{\otimes } \mhK[t]} , \ldots, \sigma_0^{(s+1)} (t) , \ubr{- q_{s+1}}_{\minc{\eqref{W14Z6}} \mmF_{2^s} L}  \}_{m+1} }_{\in \mmF_{2^{s+1}} L \hat{\otimes } \mhK[t]} \Big) \\
		&=\ubr{a_{s} - d_{L} (q_{s+1})}_{\minc{\eqref{W14Z6}} \mmF_r L} + \mcO ( \mmF_{2^{s+1}} L) \minc{2^{s+1}  > r} \mmF_r L.
		\end{aligned}
		\right.	
		\end{align}
		This proves the first statement.\\
		In case of $H^i ((\mmF_1 L)/(\mmF_q L))=0$ we need to slightly adjust our previous approach. Instead of Equation \eqref{W14Z1}, $H^0 ( (\mmF_1 L) / ( \mmF_{2^{u+1}} L))=0$ only gives
		\begin{align}
		\label{ZZ1}
		\left\{
		\begin{aligned}
		\exists q_{u+1} &\in \mmF_1 L~\mst\\
		a_u - d_L (q_{u+1}) &\in \mmF_{2^{u+1}} L.
		\end{aligned}
		\right.
		\end{align}
		However, due to $a_u \in \mmF_{2^u} L$,
		\begin{align}
		d_L (q_{u+1}) \in \mmF_{2^u} L
		\end{align}
		holds.\\
		Together with $H^{-1} ((\mmF_1 L)/(\mmF_{2^u} L))=0$ this implies
		\begin{align}
		\left\{
		\begin{aligned}
		\exists r_{u+1} &\in \mmF_1 L ~\mst\\
		q_{u+1} - d_L (r_{u+1}) &\in \mmF_{2^u} L.
		\end{aligned}
		\right.
		\end{align}
		From then on the rest of the proof remains the same, with the sole difference of using
		\begin{align*}
		\sigma_1^{(u+1)} \meqd - (q_{u+1} - d_L (r_{u+1}))
		\end{align*}
		instead.\\
		The same adjustments have also to be done when going from $a_s \in \mmF_{2^s} L$ to $a_r \in \mmF_r L$ for $2^s < r < 2^{s+1}$.
	\end{proof}
	\begin{lemma}
		\label{Lemma W14Z1}\hfill\\
		Let $L$ be an $\mmS L_\infty$ algebra endowed with a descending, bounded above and complete filtration
		\begin{align*}
		L= \mmF_1 L \supset \mmF_2 L \supset \mmF_3 L \supset \ldots
		\end{align*}
		compatible with the $\mmS L_\infty$ structure.\\
		Let us assume $H^0 ( ( \mmF_{2^q} L)/(\mmF_{\mmin(r,2^{q+1})} L))=0$ for every q with $2^q < r$.\\
		If $x\in \mmF_1 L^0$ and $d_L (x) \in \mmF_{2^k} L$ for some $2^k < r$, then
		\begin{align}
		\label{ZC1}
		\begin{aligned}
		\exists y &\in \mmF_1 L~\mst\\
		x-d_L (y) &\in \mmF_{2^k} L.
		\end{aligned}
		\end{align}
		Analogously, $x\in \mmF_1 L^0$ with $d_L (x) \in \mmF_{r} L$ results in
		\begin{align}
		\label{ZC2}
		\begin{aligned}
		\exists y &\in \mmF_1 L~\mst\\
		x-d_L (y) &\in \mmF_r L.
		\end{aligned}
		\end{align}
	\end{lemma}
	\begin{proof}[Proof of Lemma \ref{Lemma W14Z1}]
		\hfill\\
		Because of $x \in \mmF_1 L^0$, $d_L (x) \in \mmF_{2^k} L$ and $H^0 ( (\mmF_{2^0} L)/(\mmF_{2^1} L))=0$, there exists a $\tilde{y}_1 \in \mmF_1 L$ such that
		\begin{align*}
		x-d_L (\tilde{y}_1) \in \mmF_{2^1} L.
		\end{align*}  
		We set 
		\begin{align*}
		\tilde{x}_1 \meqd x- d_L (\tilde{y}_1),
		\end{align*}
		which clearly has $\tilde{x}_1 \in \mmF_{2^1} L^0$.\\
		But due to $d^2=0$,
		\begin{align*}
		d_L (\tilde{x}_1) =d_L (x) \in \mmF_{2^k} L
		\end{align*}
		still holds and so we may use $H^0 ( ( \mmF_{2^1} L)/(\mmF_{2^2} L))=0$ (now we passed to the next bigger power of 2), resulting in the existence of $\tilde{y}_2 \in \mmF_2 L$ for which
		\begin{align*}
		\tilde{x}_1 - d_L ( \tilde{y}_2) \in \mmF_{2^2} L.
		\end{align*}
		We set
		\begin{align*}
		\tilde{x}_2 \meqd \tilde{x}_1 - d_L ( \tilde{y}_2) = x- d_L ( \tilde{y}_1) - d_L ( \tilde{y}_2),
		\end{align*}
		which can be said to have $\tilde{x}_2 \in \mmF_{2^2} L^0$.\\
		Continuing with this procedure eventually leads to
		\begin{align*}
		x_k \meqd \tilde{x}_{k-1} - d_L (\tilde{y}_k) = x- d_L (\tilde{y}_k + \tilde{y}_{k-1} + \ldots + \tilde{y} ),
		\end{align*}
		satisfying $x_k \in \mmF_{2^k} L$.\\
		The second statement follows in a completely analogous manner.
	\end{proof}
	\section{Bijection on the Level of $\pi_0$}
	\label{Bijection}
	\begin{theorem}[Bijection of the connected components]
		\label{Thm W14B1}\hfill\\
		Let $L$ and $\tilde{L}$ be two $\mmS L_\infty$ algebras equipped with descending, bounded above and complete filtrations
		\begin{align*}
		\begin{aligned}
		L= \mmF_1 L \supset \mmF_2 L \supset \mmF_3 L \supset \ldots\\
		\tilde{L} = \mmF_1 \tilde{L} \supset \mmF_2 \tilde{L} \supset \mmF_3 \tilde{L} \supset \ldots
		\end{aligned}
		\end{align*}
		compatible with the $\mmS L_\infty$ algebra structures.\\
		Let $U:L \rightarrow \tilde{L}$ be an $\infty$-morphism of $\mmS L_\infty$ algebras compatible with the filtrations.\\
		Let $\psi$ be the linear term of $U$ as in Equation \eqref{W14B7} and let $U_\star$ be defined as in Equation \eqref{W14B8}.\\
		Let $\psi$ be a quasi-isomorphism on the r-1st page of the spectral sequences of the filtered complexes $(L,d_L)$ and $(\tilde{L},d_{\tilde{L}})$.\\
		Let us assume $H^0 ((\mmF_{2^q} L)/(\mmF_{\mmin(2^{q+1},r)} L))=0$ for every $q$ with $2^q < r$ and\\
		$H^i((\mmF_1 \tilde{L}) / (\mmF_q \tilde{L}))=0$ for $i=-1,0$ and $q$ every power of 2 smaller than $r$ and $q=r$.\\
		Then
		\begin{align}
		\label{W14B34}
		U_\star :\mMC(L) / \sim \stackrel{\cong}{\rightarrow} \mMC(\tilde{L}) /\sim
		\end{align}
		is a bijection, where $\sim$ denotes gauge equivalence induced by 1-cells in $\mmMC(L)$ and $\mmMC(\tilde{L})$, respectively.
	\end{theorem}
	\begin{proof}[\textbf{Proof of Theorem \ref{Thm W14B1}}]
		\hfill\\
		We show surjectivity and injectivity separately.
		\subsection{Surjectivity}
		\label{Surjectivity 1}\hfill\\
		We prove by induction on $p$ the following statement, from which surjectivity directly follows:\\
		\underline{Statement:}\\
		Let $b \in \mMC (\tilde{L})$ be arbitrary.\\
		Then there exists a sequence $\{a_p\}_{p \geq r+1}$ of degree $0$ elements in $L$, a sequence $\{ b_p\}_{p \geq r+1}$ of Maurer-Cartan elements in $\mMC(\tilde{L})$ and a sequence $\{\gamma^{(p)} \}_{p \geq r+2}$ of rectified 1-cells in $\mmMC(\tilde{L})$, such that:
		\begin{enumerate}
			\item $ b_{r+1} \sim b$.
			\item $a_p \in \mmF_r L$ and $a_p- a_{p-1} \in \mmF_{p-r} L$.
			\item $\mcurv(a_p) \in \mmF_p L$.
			\item $ b_p \in \mmF_r \tilde{L}$ and $\gamma^{(p)} = \gamma_0^{(p)} (t) + dt~\gamma_1^{(p)}$ satisfies $\gamma_0^{(p)}(0) = b_{p-1}$ and $\gamma_0^{(p)} (1) = b_p$ (thus $b_{p} \sim b_{p-1}$), as well as $\gamma_1^{(p)} \in \mmF_{p-r} \tilde{L}$.\\
			In addition, $b_p - b_{p-1} \in \mmF_{p-r} \tilde{L}$ holds for $p > 2r-1$.
			\item $U_\star (a_p) = b_p~\mmod ~\mmF_p \tilde{L}$.
		\end{enumerate}
		\underline{Proof of Statement:}\hfill\\[-20pt]
		\subsubsection{\underline{Base of Induction:}}\hfill\\
		\underline{Constructing $a_{r+1}$:}\hfill\\
		Because of Lemma \ref{Lemma W19A1}, there exists a Maurer-Cartan element $b_r \in \mMC(\tilde{L})$ of filtration degree $b_r \in \mmF_r \tilde{L}$, which is gauge equivalent to $b$.\\
		Applying the Maurer-Cartan equation on it shows
		\begin{align}
		\label{W14B54}
		d_{\tilde{L}} (b_{r}) = \ubr{- \sum_{m=2}^{\infty} \frac{1}{m!} \{ b_{r}, \ldots, b_{r}   \}_m }_{\in \mmF_{2 r} \tilde{L}} \in \mmF_{r+r} \tilde{L}.
		\end{align}
		But then Equation \eqref{W14B54} makes sure that all the requirements of Equation \eqref{W14B14a}, in the case of $p=r$, $b_{r}$ in the role of $b$ and $a$ set to zero, are satisfied.\\
		Thus, Equation \eqref{W14B14} yields:
		\begin{align}
		&\left\{
		\begin{aligned}
		\exists a &\in \mmF_1 L\\
		\exists y &\in \mmF_1 \tilde{L}
		\end{aligned}
		\right.\\
		\nonu \mst\\
		\label{W14B56}
		&\left\{
		\begin{aligned}
		d_L (a) &\in \mmF_r L\\
		\psi (a ) - d_{\tilde{L}} (y) &\in \mmF_r \tilde{L}
		\end{aligned}
		\right.\\	
		\nonumber \mand\\
		\label{W14B57}
		&\left\{
		\begin{aligned}
		d_L (a) &\in \mmF_{r+1}	L\\
		b_{r} - \psi (a) + d_{\tilde{L}} (y) &\in \mmF_{r+1} \tilde{L}.
		\end{aligned}
		\right.	
		\end{align}
		Because of Lemma \ref{Lemma W14Z1} and Equation \eqref{W14B56}, we recognise
		\begin{align*}
		\begin{aligned}
		\exists \tilde{a} &\in \mmF_1 L~\mst\\
		a- d_L (\tilde{a}) &\in \mMF_r L,
		\end{aligned}
		\end{align*}
		which, in turn, allows us to set
		\begin{align}
		\label{W14B59}
		a_{r+1} \meqd a - d_L (\tilde{a}) \in \mmF_r L.
		\end{align}
		This so-constructed $a_{r+1}$ satisfies
		\begin{align}
		\label{W14B75b}
		\mcurv (a_{r+1}) = \ubr{d_L (a_{r+1})}_{\meqc{\eqref{W14B59}} d_L(a) \minc{\eqref{W14B57}} \mmF_{r+1} L} + \ubr{\sum_{m=2}^{\infty} \frac{1}{m!} \{ \ubr{a_{r+1}}_{\minc{\eqref{W14B59}} \mmF_r L} , \ldots, a_{r+1}  \}_m }_{\in \mmF_{2r} L \subset \mmF_{r+1} L} \in \mmF_{r+1} L.
		\end{align}
		\underline{Constructing $b_{r+1}$:}\hfill\\
		Recalling the fact that $\psi$ is linear and commutes with the differential leads to
		\begin{align}
		\label{W14B60}
		\left\{
		\begin{aligned}
		\psi (a_{r+1}) &\hspace{-4pt}\meqc{\eqref{W14B59}} \psi (a ) - d_{\tilde{L}} (\psi (\tilde{a})) \meqc{\eqref{W14B57}}	b_{r} + d_{\tilde{L}} (y) - d_{\tilde{L}} (\psi (\tilde{a})) + \mcO ( \mmF_{r+1} \tilde{L}) \\
		&= b_{r} + d_{\tilde{L}} (\hat{y}) ~\mmod~\mmF_{r+1} \tilde{L},
		\end{aligned}
		\right.	
		\end{align}
		where we set
		\begin{align}
		\label{W14B61}
		\hat{y} \meqd y - \psi (\tilde{a}).
		\end{align}
		On the other hand, compatibility with filtrations shows
		\begin{align*}
		d_{\tilde{L}} (\hat{y}) \meqc{\eqref{W14B60}} \ubr{\psi (a_{r+1})}_{\minc{\eqref{W14B59}} \mmF_r \tilde{L}} - \ubr{b_{r}}_{\in \mmF_r \tilde{L}} + \mcO (\mmF_{r+1} \tilde{L} )  \in \mmF_r \tilde{L},
		\end{align*} 
		and so making use of $H^{-1}((\mmF_1 \tilde{L}) /(\mmF_r \tilde{L}))=0$ indicates that
		\begin{align}
		\label{W14B63}
		\left\{
		\begin{aligned}
		\exists \tilde{y} &\in \mmF_1 \tilde{L}~\mst\\
		\hat{y} - d_{\tilde{L}} (\tilde{y}) &\in \mmF_r \tilde{L}.
		\end{aligned}
		\right.	
		\end{align}
		We then set
		\begin{align}
		\label{W14B64}
		y_{\mnew} \meqd \hat{y} - d_{\tilde{L}} (\tilde{y})
		\end{align}
		and realise, that Equation \eqref{W14B57} still holds for $a$ replaced by $a_{r+1}$ and $y$ changed to $y_{\mnew}$, since
		\begin{align}
		\label{W14B65}
		\left\{
		\begin{aligned}
		&b_{r} - \psi (a_{r+1}) + d_{\tilde{L}} (y_{\mnew})\\
		&\hspace{-7pt}\meqc{\eqref{W14B60}} b_{r}- \left( b_{r}+ d_{\tilde{L}} (\hat{y}) + \mcO ( \mmF_{r+1}\tilde{L}) \right) + d_{\tilde{L}} (\hat{y} ) = 0 ~\mmod ~\mmF_{r+1} \tilde{L}.
		\end{aligned}
		\right.
		\end{align}
		Next, we construct a rectified 1-cell $\gamma= \gamma_0 (t) + dt~\gamma_1$ in $\mmMC(\tilde{L})$ in the usual manner by setting its starting point to $b_{r}$ and demanding
		\begin{align}
		\label{W14B66}
		\gamma_1 \meqd y_{\mnew}.
		\end{align}
		Due to Lemma \ref{Lemma W14B1}, $\gamma_0 (t)$ is found to be
		\begin{align*}
		\gamma_0 (t) = \ubr{ b_{r}}_{\in \mmF_r \tilde{L}}  + \int_0^{t} dt_1 \Big( \ubr{d_{\tilde{L}} (y_{\mnew})}_{\minc{\substack{\eqref{W14B63}\\\eqref{W14B64}}} \mmF_r \tilde{L}} + \sum_{m=1}^{\infty} \frac{1}{m!} \{ \gamma_0 (t_1), \ldots, \gamma_0 (t_1), \hspace{-3pt}\ubr{y_{\mnew}}_{\minc{\substack{\eqref{W14B63}\\\eqref{W14B64}}} \mmF_r \tilde{L}} \hspace{-3pt} \}_{m+1} \Big),  
		\end{align*}
		i.e.
		\begin{align}
		\label{W14B68}
		\gamma_0 (t) \in \mmF_r \tilde{L} \hat{\otimes } \mhK[t].
		\end{align}
		We set $b_{r+1}$ to be the endpoint of this 1-cell, i.e.
		\begin{align}
		\label{W14B69}
		b_{r+1} \meqd \gamma \vert_{t=1} = \gamma_0 (1).
		\end{align}
		Thus, it satisfies
		\begin{align}
		\label{W14B70}
		\left\{
		\begin{aligned}
		&b_{r+1} \meqd  \gamma_0 (1)\\
		&=b_{r} + \int_0^1 dt \Big( d_{\tilde{L}} (y_{\mnew}) + \ubr{\sum_{m=1}^{\infty} \frac{1}{m!} \{ \hspace{-6pt} \ubr{\gamma_0 (t)}_{\minc{\eqref{W14B68}} \mmF_r \tilde{L} \hat{\otimes} \mhK [t]} \hspace{-6pt}, \ldots, \gamma_0 (t), \hspace{-3pt}\ubr{y_{\mnew}}_{\minc{\substack{\eqref{W14B63}\\\eqref{W14B64}}} \mmF_r \tilde{L}} \hspace{-3pt} \}_{m+1} }_{\in \mmF_{2r} \tilde{L}\hat{\otimes } \mhK [t] \subset \mmF_{r+1} \tilde{L}\hat{\otimes } \mhK [t]} \Big)\\
		&=\ubr{b_{r}}_{\in \mmF_r \tilde{L}} + \ubr{d_{\tilde{L}} (y_{\mnew})}_{\minc{\substack{\eqref{W14B63}\\\eqref{W14B64}}} \mmF_r \tilde{L}} + \mcO ( \mmF_{r+1} \tilde{L}),
		\end{aligned}
		\right.	
		\end{align}
		so particularly
		\begin{align}
		\label{W14B71}
		b_{r+1} \in \mmF_r \tilde{L}
		\end{align}
		holds.\\
		\underline{Proving $U_\star (a_{r+1}) = b_{r+1}~\mmod~ \mmF_{r+1} \tilde{L}$:}\hfill\\
		Due to the definition of $U_\star$ (cf. Equation \eqref{W14B8}) and $a_{r+1} \minc{\eqref{W14B59}} \mmF_r L$, it is obvious that
		\begin{align}
		\label{W14B73}
		U_\star (a_{r+1}) = \psi (a_{r+1}) + \mcO (\mmF_{r+1} \tilde{L}).
		\end{align}
		Hence, collecting Equations \eqref{W14B65}, \eqref{W14B70} and \eqref{W14B73} eventually results in
		\begin{align*}
		U_\star (a_{r+1}) - b_{r+1} \meqc{\substack{\eqref{W14B70} \\ \eqref{W14B73}}} \ubr{ \psi (a_{r+1}) - \left( b_{r} + d_{\tilde{L}} (y_{\mnew}) \right)}_{\minc{\eqref{W14B65}} \mmF_{r+1} \tilde{L}} + \mcO(\mmF_{r+1} \tilde{L} ),
		\end{align*}\vspace{-4pt}
		i.e. \vspace{-1pt}
		\begin{align}
		\label{W14B75}
		U_\star (a_{r+1}) = b_{r+1} ~\mmod ~\mmF_{r+1} \tilde{L}.
		\end{align}
		\subsubsection{\underline{Induction Step:}}\hfill\\
		\underline{Preparation:}\hfill\\
		We start by setting
		\begin{align}
		\label{W14B76}
		\tilde{b} \meqd U_\star (a_p)- b_p \in \mmF_p \tilde{L},
		\end{align}
		which holds due to the assumption.\\
		We observe that
		\begin{align*}
		d_L (\mcurv (a_p)) + \ubr{\sum_{k=1}^{\infty} \frac{1}{k!} \{ \ubr{a_p}_{\in \mmF_r L}, \ldots, a_p , \ubr{\mcurv(a_p)}_{\in \mmF_p L}  \}_{k+1} }_{\in \mmF_{p+r} L} \meqc{\text{(2.17) from \cite{DolgushevEnhancement}}} 0,
		\end{align*}
		i.e.
		\begin{align}
		\label{W14B79}
		d_L (\mcurv (a_p)) = 0 ~\mmod~\mmF_{p+r} L
		\end{align}
		holds.\vspace{5pt}\\
		We continue by computing
		\begin{align*}
		\begin{aligned}
		\psi (\mcurv (a_p)) &\meqc{\text{(2.18) from \cite{DolgushevEnhancement}}} \mcurv (U_\star (a_p)) - \ubr{\sum_{m=1}^{\infty} \frac{1}{m!} U^\prime (\ubr{\ubr{a_p}_{\in \mmF_r L}, \ldots, a_p}_{\text{m-times}}, \ubr{\mcurv (a_p)}_{\in \mmF_p L})}_{\in \mmF_{p+r} \tilde{L}}\\
		&\hspace{21pt}=\nonu \mcurv (U_\star (a_p)) + \mcO (\mmF_{p+r} \tilde{L}) \meqc{\eqref{W14B76}} \mcurv (\tilde{b} + b_p) + \mcO (\mmF_{p+r} \tilde{L}) \\
		& \meqc{\text{(2.20) from \cite{DolgushevEnhancement}} } \nonu \ubr{\mcurv(b_p)}_{\meqc{b_p \in \mMC(\tilde{L})} 0} + d_{\tilde{L}}^{b_p} (\tilde{b}) + \sum_{m=2}^{\infty} \frac{1}{m!} \{\tilde{b}, \ldots, \tilde{b}\}_m^{b_p} + \mcO (\mmF_{p+r} \tilde{L})\\
		&\hspace{21pt}=\nonu d_{\tilde{L}} (\tilde{b}) + \ubr{ \sum_{k=1}^{\infty} \frac{1}{k!} \{ \ubr{b_p}_{\in \mmF_r \tilde{L} } ,\ldots, b_p, \ubr{\tilde{b}}_{\minc{\eqref{W14B76}} \mmF_p \tilde{L}} \}_{k+1} }_{\in \mmF_{p+r} \tilde{L}}\\
		&\hspace{21pt}+\ubr{\sum_{m=2}^{\infty} \frac{1}{m!} \sum_{k=0}^{\infty} \frac{1}{k!} \{ \ubr{b_p, \ldots, b_p}_{\text{k-times}} , \ubr{\ubr{\tilde{b}}_{\minc{\eqref{W14B76}} \mmF_p \tilde{L}}, \ldots, \tilde{b}}_{\text{m-times}}  \}_{k+m}}_{\in \mmF_{2p} \tilde{L} \underset{p \geq r+1}{\subset} \mmF_{p+r} \tilde{L}} + \mcO ( \mmF_{p+r} \tilde{L})\\
		&\hspace{21pt}=\nonu d_{\tilde{L}} (\tilde{b}) + \mcO (\mmF_{p+r} \tilde{L}),
		\end{aligned}	
		\end{align*}
		i.e.
		\begin{align}
		\label{W14B81}
		\psi(\mcurv(a_p))- d_{\tilde{L}} (\tilde{b}) \in \mmF_{p+r} \tilde{L}.
		\end{align}
		Due to Equations \eqref{W14B79} and \eqref{W14B81}, all the requirements of Equation \eqref{W14B14a}, in the setting of $p=p$, $\mcurv(a_p)$ playing the role of $a$ and $\tilde{b}$ playing the role of $b$, are satisfied.\\
		Thus, Equation \eqref{W14B14} implies 
		\begin{align}
		\label{W14B82}
		&\left\{
		\begin{aligned}
		\exists \tilde{a} &\in \mmF_{p-r+1} L\\
		\exists y &\in \mmF_{p-r+1} \tilde{L}
		\end{aligned}
		\right.\\
		\nonu \mst\\
		\label{W14B83}
		&\left\{
		\begin{aligned}
		d_L (\tilde{a}) &\in \mmF_p L\\
		\psi (\tilde{a}) - d_{\tilde{L}} (y) &\in \mmF_p \tilde{L}
		\end{aligned}
		\right.\\	
		\nonu \mand\\
		\label{W14B84}
		&\left\{
		\begin{aligned}
		\mcurv (a_p) - d_L (\tilde{a}) &\in \mmF_{p+1} L\\
		\tilde{b} - \psi (\tilde{a}) + d_{\tilde{L}} (y) &\in \mMF_{p+1} \tilde{L}.	
		\end{aligned}
		\right.		
		\end{align}
		We continue by distinguishing the two cases $r+1 \leq p < 2r-1$ and $p \geq 2r-1$.
		\underline{\textbf{Case 1: $p \geq 2r-1$:}}\hfill\\
		\underline{Constructing $a_{p+1}$:}\hfill\\
		We set
		\begin{align}
		\label{W14B85}
		a_{p+1} \meqd a_p - \tilde{a} \in \mmF_r L,
		\end{align}
		by the virtue of Equation \eqref{W14B82} and $p \geq 2r-1$.\\[3pt]
		Further, we compute the degree of filtration of $\mcurv(a_{p+1})$ via
		\begin{align}
		\label{W14B86}
		\left\{
		\begin{aligned}
		&\hspace{21pt}\mcurv (a_{p+1})\meqc{\eqref{W14B85}} \mcurv (a_p - \tilde{a})\\
		&\meqc{\text{(2.20) from \cite{DolgushevEnhancement}}} \mcurv(a_p) + d_L^{a_p} (- \tilde{a}) + \sum_{m=2}^{\infty} \frac{1}{m!} \{ - \tilde{a}, \ldots, - \tilde{a}\}_m^{a_p} \\
		&\hspace{21pt}=\ubr{\mcurv(a_p) + d_L (- \tilde{a})}_{\minc{\eqref{W14B84}} \mmF_{p+1} L} + \ubr{\sum_{k=1}^{\infty} \frac{1}{k!} \{\ubr{a_p}_{\in \mmF_r L}, \ldots , a_p,  \ubr{- \tilde{a}}_{\minc{\eqref{W14B82}} \mmF_{p-r+1} L} \}_{k+1}}_{\in \mmF_{r+(p-r+1)} L = \mmF_{p+1} L}\\
		&\hspace{21pt}+\ubr{\sum_{m=2}^{\infty} \frac{1}{m!} \sum_{k=0}^{\infty}\frac{1}{k!} \{\ubr{a_p, \ldots, a_p}_{\text{k-times}} , \ubr{\ubr{- \tilde{a}}_{\minc{\eqref{W14B82}} \mmF_{p-r+1} L}, \ldots,-\tilde{a}}_{\text{m-times}} \}_{m+k}}_{\in \mmF_{2(p-r+1)} L {\subset} \mmF_{p+1} L} \in \mmF_{p+1} L.
		\end{aligned}
		\right.	
		\end{align}
		\underline{Constructing $b_{p+1}$:}\hfill\\
		Next, we define a rectified 1-cell $\gamma^{(p+1)} = \gamma_0^{(p+1)} (t) + dt~\gamma_1^{(p+1)}$ in $\mmMC(\tilde{L})$ by setting its starting point to $b_p$ and demanding 
		\begin{align}
		\label{W14B87}
		\gamma_1^{(p+1)} \meqd -y.
		\end{align}
		Integrating the corresponding IVP (cf. Lemma \ref{Lemma W14B1}) and analysing its degree of filtration yields
		\begin{align}
		\label{W14B88}
		\gamma_0^{(p+1)} (t) \in \mmF_r \tilde{L} \hat{\otimes } \mhK[t].	
		\end{align}
		We set $b_{p+1}$ to be the endpoint, i.e.
		\begin{align}
		\label{W14B89}
		b_{p+1} \meqd \gamma^{(p+1)} \vert_{t=1} = \gamma_0^{(p+1)} (1) \minc{\eqref{W14B88}} \mmF_r \tilde{L},
		\end{align}
		and compute
		\begin{align}
		\label{W14B90}
		\left\{
		\begin{aligned}
		&b_{p+1} \meqd \gamma_0^{(p+1)} (1)\\
		&=b_p + \int_{0}^{1}\hspace{-3pt} dt \Big( d_{\tilde{L}} (-y) + \ubr{\sum_{m=1}^{\infty} \frac{1}{m!} \{ \ubr{\gamma_0^{(p+1)} (t)}_{\minc{\eqref{W14B88}} \mmF_r \tilde{L} \hat{\otimes} \mhK [t]} , \ldots, \gamma_0^{(p+1)}(t),\hspace{-15pt} \ubr{-y}_{\minc{\eqref{W14B82}} \mmF_{p-r+1} \tilde{L}}\hspace{-13pt}  \}_{m+1}}_{\in \mmF_{p+1} \tilde{L} \hat{\otimes } \mhK[t]} \Big)\\[-8pt]
		&=b_p - d_{\tilde{L}} (y) + \mcO(\mmF_{p+1} \tilde{L}).
		\end{aligned}
		\right.
		\end{align}
		\underline{Proving $U_\star (a_{p+1}) = b_{p+1}~\mmod~ \mmF_{p+1} \tilde{L}$:}\hfill\\
		The difference between $U_\star (a_{p+1})$ and $b_{p+1}$ is found to be
		\begin{align*}
		\begin{aligned}
		U_\star (a_{p+1}) - b_{p+1} &\meqc{\substack{\eqref{W14B85}\\\eqref{W14B90}}}
		\ubr{-\psi (\tilde{a})+ \ubr{U_\star (a_p) - b_p}_{\meqc{\eqref{W14B76}} \tilde{b}} + d_{\tilde{L}} (y)}_{\minc{\eqref{W14B84}} \mmF_{p+1} \tilde{L}}\\
		&\hspace{5pt}+\bigg( \sum_{m=2}^{\infty} \frac{1}{m!} U^\prime \big(\ubr{(a_p - \tilde{a}), \ldots, (a_p - \tilde{a})}_{\text{m-times}} \big) + \psi (a_p) - U_\star (a_p)  \bigg) + \mcO (\mmF_{p+1} \tilde{L}).
		\end{aligned}
		\end{align*}
		By unravelling the definition of $U_\star$ and repeated use of the multilineality of $U^\prime$, we recognise the bracket lying in $\mmF_{p+1} \tilde{L}$.\\
		Thus, we may deduce
		\begin{align}
		\label{W14B92}
		U_\star (a_{p+1}) = b_{p+1} ~\mmod~\mmF_{p+1} \tilde{L}
		\end{align}
		to hold.\\
		\underline{\textbf{Case 2: $r+1 \leq p < 2r-1$:}}\hfill\\
		\underline{Constructing $a_{p+1}$:}\hfill\\
		Equation \eqref{W14B83} together with Lemma \ref{Lemma W14Z1} implies
		\begin{align}
		\label{W14B93}
		\left\{
		\begin{aligned}
		\exists \hat{a} &\in \mmF_1 L~\mst\\
		\tilde{a} - d_L ( \hat{a}) &\in \mmF_r L.
		\end{aligned}
		\right.	
		\end{align}
		We then set
		\begin{align}
		\label{W14B94}	
		a_{p+1}  \meqd \ubr{a_p}_{\in \mmF_r L} - \big( \ubr{\tilde{a} - d_L ( \hat{a})}_{\minc{\eqref{W14B93}} \mmF_r L} \big) \in \mmF_r L.
		\end{align}
		Similar calculations as in Equation \eqref{W14B86} prove 
		\begin{align}
		\label{W14B105}
		\mcurv (a_{p+1}) \in \mmF_{p+1} \tilde{L}
		\end{align}
		to hold.\\
		\underline{Constructing $b_{p+1}$:}\hfill\\
		It is clear from Equations \eqref{W14B83} and \eqref{W14B93} that
		\begin{align}
		\label{W14B95}
		d_{\tilde{L}} (y - \psi (\hat{a})) \in \mmF_r \tilde{L}.
		\end{align}
		Setting
		\begin{align}
		\label{W14B96}
		\hat{y} \meqd y - \psi ( \hat{a})
		\end{align}
		and using $H^{-1}((\mmF_1 \tilde{L})/(\mmF_r \tilde{L}))=0$ leads to
		\begin{align*}
		\begin{aligned}
		\exists q &\in \mmF_1 \tilde{L}~\mst\\
		\hat{y} - d_{\tilde{L}} (q) &\in \mmF_r \tilde{L}.
		\end{aligned}
		\end{align*}
		Let $\gamma^{(p+1)}$ be the rectified 1-cell $\gamma^{(p+1)} = \gamma_0^{(p+1)} (t) + dt~\gamma_1^{(p+1)}$ in $\mmMC(\tilde{L})$ determined by the starting point $b_p$ and the condition
		\begin{align}
		\label{W14B98}
		\gamma_1^{(p+1)} \meqd - ( \hat{y} - d_{\tilde{L}} (q)).
		\end{align}
		Once more, we construct $b_{p+1}$ as the endpoint of this 1-cell, i.e.
		\begin{align}
		\label{W14B99}
		b_{p+1} \meqd \gamma_0^{(p+1)}(1).
		\end{align}
		According to Lemma \ref{Lemma W14B1}, $\gamma_0^{(p+1)} (t)$ is the solution of the corresponding IVP. Integration results in
		\begin{align}
		\label{W14B100}
		\gamma_0^{(p+1)} (t) \in \mMF_r \tilde{L} \hat{\otimes} \mhK[t].
		\end{align}
		This can then be used for estimating the degree of filtration in the expression of $b_{p+1}$ and so we get
		\begin{align}
		\label{W14B101}
		b_{p+1} = \ubr{b_p}_{\in \mmF_r \tilde{L}} - \ubr{d_{\tilde{L}} (\hat{y})}_{\minc{\substack{\eqref{W14B95}\\\eqref{W14B96}}} \mmF_r \tilde{L}} + \mcO(\mmF_{2r} \tilde{L}).
		\end{align}
		This ensures
		\begin{align}
		\label{W14B102}
		b_{p+1} \in \mmF_r \tilde{L}.
		\end{align}
		\underline{Proving $U_\star (a_{p+1})=b_{p+1}~\mmod ~
			\mmF_{p+1} \tilde{L}$:}\hfill\\
		Equation \eqref{W14B101} together with the definitions of $a_{p+1}$ and $U_\star$ allows us to calculate
		\begin{align*}
		\begin{aligned}
		&U_\star (a_{p+1}) - b_{p+1} \meqc{\substack{\eqref{W14B94}\\\eqref{W14B96}\\\eqref{W14B101}}} \ubr{- \psi (\tilde{a}) + \ubr{U_\star (a_p) - b_p}_{\meqc{\eqref{W14B76}} \tilde{b}} + d_{\tilde{L}} (y)}_{\minc{\eqref{W14B84}} \mmF_{p+1} \tilde{L}} + \ubr{d_{\tilde{L}} \big( \psi (\hat{a}) \big) - d_{\tilde{L}} \big( \psi (\hat{a}) \big)}_{=0} \\[-6pt]
		&+\bigg( \sum_{m=2}^{\infty} \frac{1}{m!} U^\prime \big(\ubr{a_p - (\tilde{a}- d_L (\hat{a})), \ldots, a_p - (\tilde{a}- d_L (\hat{a}))}_{\text{m-times}} \big) + \psi (a_p) - U_\star (a_p)  \bigg) + \ubr{\mcO(\mmF_{2r} \tilde{L})}_{{\subset} \mmF_{p+1} \tilde{L}}.
		\end{aligned}
		\end{align*}
		A closer investigation reveals the bracket to lie in $\mmF_{p+1} \tilde{L}$ and so we get
		\begin{align}
		\label{W14B104}
		U_\star (a_{p+1}) = b_{p+1} ~\mmod~\mmF_{p+1} \tilde{L}.
		\end{align}
		\subsection{Injectivity}\hfill\\
		By induction on $p$, we assert the following statement, from which injectivity follows:\\
		\underline{Statement:}\\
		Let $a \in \mMC(L)$ be an arbitrary Maurer-Cartan element which satisfies $U_\star (a) \sim 0$.\\
		Then there exists a sequence $\{a_p\}_{p \geq r}$ of Maurer-Cartan elements in $\mMC(L)$ and sequences $\{\gamma^{(p)}\}_{p \geq r}$ and $\{\xi^{(p)}\}_{p \geq r+2}$ of rectified 1-cells in $\mmMC(\tilde{L})$ and $\mmMC(L)$, respectively, such that:
		\begin{enumerate}
			\item $a_r \sim a$.
			\item $a_p \in \mmF_p L$.
			\item $\xi^{(p)} = \xi_0^{(p)} (t) + dt~\xi_1^{(p)}$ satisfies $\xi_0^{(p)} (0)= a_{p-1}$ and $\xi_0^{(p)} (1)= a_p$ (thus $a_p \sim a_{p-1}$), as well as $\xi_1^{(p)} \in \mmF_{p-r} L$.
			\item $ \gamma^{(p)} = \gamma_0^{(p)} (t) + dt~\gamma_1^{(p)}$ satisfies $\gamma_0^{(p)}  (0)=0$ and $ \gamma_0^{(p)} (1) = U_\star (a_p)$ (thus $0 \sim U_\star (a_p)$).\\
			In addition, $\gamma_1^{(p)} \in \mmF_p \tilde{L}$ holds.
		\end{enumerate}
		\underline{Proof of Statement:}\hfill\\[-20pt]
		\subsubsection{\underline{Base of induction:}}\hfill\\
		Due to Lemma \ref{Lemma W19A1}, there exists a Maurer-Cartan element $a_r \in \mMC(L)$ which has filtration degree $\mmF_r L$ and is gauge equivalent to $a$.\\
		This particularly implies
		\begin{align*}
		U_\star (a_r) \sim U_\star (a) \sim 0
		\end{align*}
		to hold.\\
		Therefore, there exists a 1-cell $\gamma^{(1)} = \gamma_0^{(1)} (t) + dt~\gamma_1^{(1)}$ in $\mmMC(\tilde{L})$ (due to Lemma B2 from \cite{DolgushevLinfty} we can w.l.o.g. assume $\gamma^{(1)}$ to be rectified) with some $\gamma_1^{(1)}\in \mmF_1 \tilde{L}$ such that
		\begin{align*}
		\gamma_0^{(1)} (0) = 0
		\end{align*}
		and
		\begin{align*}
		\gamma_0^{(1)} (1) = U_\star (a_r) \in \mmF_r \tilde{L}.
		\end{align*}
		According to Lemma \ref{Lemma W14B1}, for $\gamma^{(1)}$ to be a 1-cell in $\mmMC(\tilde{L})$ connecting $0$ with $U_\star (a_r)$, it must solve
		\begin{align*}
		\begin{aligned}
		\frac{\partial \gamma_0^{(1)} (t)}{\partial t} &= d_{\tilde{L}}^{\gamma_0^{(1)} (t)} (\gamma_1^{(1)})\\
		\gamma_0^{(1)} (0) &= 0\\
		\gamma_0^{(1)} (1) &= U_\star (a_r).
		\end{aligned}
		\end{align*}
		On one hand, integration up to some general $t$ yields
		\begin{align}
		\label{W14B111}
		\gamma_0^{(1)} (t) = 0 + t d_{\tilde{L}} (\gamma_1^{(1)}) + \int_{0}^{t} dt_1~\sum_{k=1}^{\infty} \frac{1}{k!} \{\gamma_0^{(1)}(t_1), \ldots, \gamma_0^{(1)}(t_1), \gamma_1^{(1)} \}_{k+1}.
		\end{align}
		On the other hand, we can also integrate up to $t=1$ and replace $\gamma_0^{(1)} (1)$ by $U_\star (a_r)$ and find
		\begin{align}
		\label{W14B112}
		\ubr{U_\star (a_r)}_{\in \mmF_r \tilde{L}} = 0 + d_{\tilde{L}} (\gamma_1^{(1)}) + \int_{0}^{1} dt~\sum_{k=1}^{\infty} \frac{1}{k!} \{ \gamma_0^{(1)} (t), \ldots, \gamma_0^{(1)} (t), \gamma_1^{(1)}  \}_{k+1}.
		\end{align}
		From the expression in the latter integral at least consisting of a $\{.,.\}_2$ bracket, it is obvious that
		\begin{align*}
		d_{\tilde{L}} (\gamma_1^{(1)}) \in \mmF_2 \tilde{L}.
		\end{align*}
		In turn, this can be plugged into Equation \eqref{W14B111}, where by the same arguments the r.h.s. is of filtration degree $\mmF_2 \tilde{L} \hat{\otimes } \mhK[t]$, hence
		\begin{align*}
		\gamma_0^{(1)} (t) \in \mmF_2 \tilde{L} \hat{\otimes } \mhK[t].
		\end{align*}
		But this may then again be used in Equation \eqref{W14B112} and so on and so forth.\\
		This procedure can be run several times until eventually we arrive at
		\begin{align}
		\label{W14B116}
		d_{\tilde{L}} (\gamma_1^{(1)}) \in \mmF_r \tilde{L}.
		\end{align}
		As it turns out, Equation \eqref{W14B116} is a good starting point for repeated application of Lemma \ref{Lemma W14B2}.\\
		By assumption we have $H^{-1}((\mmF_1 \tilde{L}) /(\mmF_r \tilde{L})) =0$. Together with Equation \eqref{W14B116}, this implies
		\begin{align*}
		\begin{aligned}
		\exists \hat{y}^{(1)} &\in \mmF_1 \tilde{L}~\mst\\
		\gamma_1^{(1)} - d_{\tilde{L}} (\hat{y}^{(1)}) &\in \mmF_r \tilde{L}.
		\end{aligned}
		\end{align*}
		It is clear, that all the requirements of Lemma \ref{Lemma W14B2} with $\gamma^{(1)}$ in the role of $\beta$ connecting the two Maurer-Cartan elements $0$ and $U_\star (a_r)$, both lying in $\mmF_r \tilde{L}$ (so $q = r \geq 1$), $\gamma_1^{(1)} \in \mmF_1 \tilde{L}$ (so $p=1$) and $\hat{y}^{(1)} \in \mmF_1 \tilde{L}$ (so $l=0$) acting as $y$ in the notation of the lemma, are satisfied.\\
		Thus, there exists a rectified 1-cell $\gamma^{(2)}= \gamma_0^{(2)} (t) +dt~\gamma_1^{(2)}$ in $\mmMC(\tilde{L})$ which connects $0$ with $U_\star (a_r)$ and has $\gamma_1^{(2)} \in \mmF_2 \tilde{L}$.\\
		We may replace $\gamma^{(1)}$ by $\gamma^{(2)}$ and perform the same steps as before. At this end, we arrive at the analogue of Equation \eqref{W14B116}, i.e. $d_{\tilde{L}} (\gamma_1^{(2)}) \in \mmF_r \tilde{L}$.\\
		Exploiting $H^{-1}((\mmF_1 \tilde{L}) /( \mmF_r \tilde{L}))=0$ leads to
		\begin{align*}
		\begin{aligned}
		\exists \hat{y}^{(2)} &\in \mmF_1 \tilde{L}~\mst\\
		\gamma_1^{(2)} - d_{\tilde{L}} (\hat{y}^{(2)}) &\in \mmF_r \tilde{L}.
		\end{aligned}
		\end{align*}
		But this time we know $\gamma_1^{(2)} \in \mmF_2 \tilde{L}$, wherefore we can apply Lemma \ref{Lemma W14B2} with $p=2$ instead.\\
		So, after several runs, we eventually arrive at a rectified 1-cell $\gamma^{(r)} = \gamma_0^{(r)} (t) + dt~\gamma_1^{(r)}$ in $\mmMC(\tilde{L})$ which connects $0$ with $U_\star (a_r)$ and has $\gamma_1^{(r)} \in \mmF_r \tilde{L}$.\\[-15pt]
		\subsubsection{\underline{Induction Step:}}\hfill\\
		\underline{Construction of $a_{p+1}$:}\hfill \\
		From $a_p \in \mmF_p L$ and $a_p \in \mMC(L)$, we may deduce 
		\begin{align}
		\label{W14B120a}
		d_L (a_p) \in \mmF_{p+r} L
		\end{align}
		in the usual manner.\\
		$\gamma^{(p)}$ is known to be a rectified 1-cell in $\mmMC(\tilde{L})$ connecting $0$ with $U_\star (a_p)$ and having $\gamma_1^{(p)} \in \mmF_p \tilde{L}$.\\
		According to Lemma \ref{Lemma W14B1}, $\gamma_0^{(p)} (t)$ can be seen as the solution of an IVP. Integration up to some arbitrary $t$ and a short analysis of the degree of filtration implies
		\begin{align}
		\label{W14B120}
		\gamma_0^{(p)} (t) \in \mmF_p \tilde{L} \hat{\otimes } \mhK[t].
		\end{align}
		By the definition of $U_\star$ and because of $a_p \in \mmF_p L$, clearly $U_\star (a_p) = \psi (a_p) + \mcO( \mmF_{p+r} \tilde{L})$ has to hold.\\
		Moreover, we have $U_\star (a_p)- \gamma_0^{(p)} (1)=0$ from $U_\star (a_p)$ being the endpoint of $\gamma^{(p)}$.\\
		Rewriting $\gamma_0^{(p)} (1)$ as the integral of the corresponding IVP and expressing $U_\star (a_p)$ in terms of $\psi (a_p)$ yields
		\begin{align*}
		\begin{aligned}
		\nonu 0 &= \left(\psi (a_p) + \mcO(\mmF_{p+r} \tilde{L} ) \right) - \Bigg( 0 + \int_{0}^{1} dt \Big( d_{\tilde{L}} (\gamma_1^{(p)}) + \ubr{\sum_{m=1}^{\infty} \frac{1}{m!} \{ \hspace{-8pt} \ubr{ \gamma_0^{(p)} (t)}_{\minc{\eqref{W14B120}} \mmF_p \tilde{L} \hat{\otimes } \mhK[t] } \hspace{-6pt}, \ldots, \gamma_0^{(p)} (t), \ubr{\gamma_1^{(p)}}_{\in \mmF_p \tilde{L}} \}_{m+1}   }_{\in \mmF_{2p} \tilde{L} \hat{\otimes} \mhK[t] \underset{p\geq r}{\subset} \mmF_{p+r} \tilde{L} \hat{\otimes } \mhK[t]} \Big) \Bigg)\\[-1cm]
		&=\psi (a_p) - d_{\tilde{L}} (\gamma_1^{(p)}) + \mcO(\mmF_{p+r} \tilde{L}),
		\end{aligned}	
		\end{align*}
		i.e.
		\begin{align}
		\label{W14B122}
		\psi (a_p) - d_{\tilde{L}} (\gamma_1^{(p)}) \in \mmF_{p+r} \tilde{L}.
		\end{align}
		Due to Equations \eqref{W14B120a} and \eqref{W14B122}, all the requirements of Equation \eqref{W14B14a} are satisfied and so we infer from Equation \eqref{W14B14} that
		\begin{align}
		\label{W14B123}
		&\left\{
		\begin{aligned}
		\exists x &\in \mmF_{p-r+1} L\\
		\exists \tilde{y} &\in \mmF_{p-r+1} \tilde{L}
		\end{aligned}
		\right.\\	
		\nonu \mst\\
		\label{W14B124}
		&\left\{
		\begin{aligned}
		d_L (x) &\in \mmF_p L\\
		\psi (x) - d_{\tilde{L}} (\tilde{y}) &\in \mmF_p \tilde{L}
		\end{aligned}
		\right.\\	
		\nonu \mand\\
		\label{W14B125}
		&\left\{
		\begin{aligned}
		a_p - d_L (x) &\in \mmF_{p+1} L\\
		\gamma_1^{(p)} - \psi (x) + d_{\tilde{L}} ( \tilde{y}) &\in \mmF_{p+1} \tilde{L}.
		\end{aligned}
		\right.	
		\end{align}
		We construct a rectified 1-cell $\xi^{(p+1)} = \xi_0^{(p+1)} (t) + dt~\xi_1^{(p+1)}$ in $\mmMC(L)$ by means of setting the starting point to $a_p$ and requiring
		\begin{align}
		\label{W14B126}
		\xi_1^{(p+1)} \meqd -x \minc{\eqref{W14B123}} \mmF_{p-r+1} L.
		\end{align}
		Integration of the corresponding IVP (cf. Lemma \ref{Lemma W14B1}) results in
		\begin{align}
		\label{W14B127}
		\xi_0^{(p+1)} (t) = \ubr{a_p}_{\in \mmF_p L}  + \int_{0}^{t} dt_1 \Big( \ubr{ d_L (-x)}_{\minc{\eqref{W14B124}} \mmF_p L} + \sum_{m=1}^{\infty} \frac{1}{m!} \{ \xi_0^{(p+1)}(t_1) , \ldots, \xi_0^{(p+1)} (t_1), \hspace{-12pt} \ubr{-x}_{\minc{\eqref{W14B123}} \mmF_{p-r+1} L} \hspace{-8pt} \}_{m+1} \Big).
		\end{align}
		Plugging ${\xi_0^{(p+1)} (t) \in \mmF_1 L \hat{\otimes} \mhK[t]}$ into Equation \eqref{W14B127} directly leads to ${\xi_0^{(p+1)} (t) \in \mmF_{p-r+2} L \hat{\otimes} \mhK[t]}$. But this can again be plugged into Equation \eqref{W14B127}, resulting in ${\xi_0^{(p+1)} (t) \in \mmF_{2p-2r+3} L \hat{\otimes } \mhK[t]}$ et cetera.\\
		We continue in the same manner and eventually arrive at
		\begin{align}
		\label{W14B128}
		\xi_0^{(p+1)} (t) \in \mmF_p L \hat{\otimes } \mhK[t].
		\end{align}
		Further, we set $a_{p+1}$ to be the endpoint of $\xi^{(p+1)}$, i.e.
		\begin{align}
		\label{W14B129}
		a_{p+1} \meqd \xi_0^{(p+1)} (1).
		\end{align}
		Of course, this can also be rewritten as the integral of the IVP, so we find
		\begin{align*}
		\begin{aligned}
		a_{p+1} &\meqd \xi_0^{(p+1)} (1)\\
		&=a_p + \int_{0}^{1} dt \Big( d_L (-x) + \ubr{\sum_{m=1}^{\infty} \frac{1}{m!} \{ \ubr{\xi_0^{(p+1)} (t)}_{\minc{\eqref{W14B128}} \mmF_p L \hat{\otimes } \mhK[t]} , \ldots, \xi_0^{(p+1)}(t) , \ubr{-x}_{\minc{\eqref{W14B123}} \mmF_{p-r+1} L}  \}_{m+1} }_{\in \mmF_{p+(p-r+1)} L \hat{\otimes } \mhK[t] \underset{p \geq r}{\subset} \mmF_{p+1} L \hat{\otimes } \mhK[t]} \Big) \\
		&=\ubr{a_p - d_L (x)}_{\minc{\eqref{W14B125}} \mmF_{p+1} L} + \mcO(\mmF_{p+1} L) \in \mmF_{p+1} L,
		\end{aligned}
		\end{align*}
		i.e.
		\begin{align}
		\label{W14B130}
		a_{p+1} \in \mmF_{p+1} L.
		\end{align}
		\underline{Connecting $0$ and $U_\star (a_{p+1})$:}\\
		Next, we construct a 1-cell in $\mmMC(\tilde{L})$ which connects $0$ and $U_\star(a_{p+1})$.
		For doing so, we apply Lemma \ref{Lemma W19A2} with $\xi^{(p+1)}$ in the role of $\xi$ and $\beta = \beta_0 (t) +dt~ \beta_1$ given by
		\begin{align}
		\left\{
		\begin{aligned}
		\beta_0(t) &\meqd \gamma_0^{(p)} (1-t)\\[-3pt]
		\beta_1 &\meqd - \gamma_1^{(p)}.
		\end{aligned}
		\right.	
		\end{align}
		As a result, we obtain a 1-cell $\kappa = \kappa_0 (t) +dt~ \kappa_1 (t)$ in $\mmMC(\tilde{L})$ connecting $0$ and $U_\star (a_{p+1})$ satisfying
		\begin{align}
		\label{W19A3}
		\kappa_1 (t) = \gamma_1^{(p)} - \psi (x) + \mcO (\mmF_{p+1} \tilde{L} \hat{\otimes } \mhK[t] ).
		\end{align}
		\underline{Adjusting the degree of filtration of the connection of $0$ and $U_\star (a_{p+1})$:}\\
		We continue by adjusting the $dt$ part of $\kappa_0(t)+ dt~\kappa_1 (t)$ by $d_{\tilde{L}} (\tilde{y})$ with $\tilde{y}$ as in Equation \eqref{W14B123} (in the sense of Lemma \ref{Lemma W14B2}). By means of Equations \eqref{W14B125} and \eqref{W19A3}, such a 1-cell would clearly have its $dt$ coefficient carrying filtration degree $\mmF_{p+1} \tilde{L}\hat{\otimes }\mhK[t]$.\\
		Using the fact that both the starting and the end point of $\kappa$ lie in $\mmF_{p+1} \tilde{L}$, we can follow the same ideas as in Lemma \ref{Lemma W14B2} even though $\kappa_1 (t)$ can only be said to be of filtration degree $\mmF_{p-r+1} \tilde{L} \hat{\otimes } \mhK[t]$ by the means of Equation \eqref{W19A3} and ${\psi(x) \hspace{-3pt}\minc{\eqref{W14B123}}\hspace{-3pt} \mmF_{p-r+1} \tilde{L}}$.\\
		Because of Equation \eqref{W19A3} we can rewrite $\kappa_1 (t)$ as
		\begin{align}
		\label{W14B157}
		\kappa_1 (t) = s+ q(t)
		\end{align}
		with some 
		\begin{align}
		\label{W14B158}
		q(t) \in \mmF_{p+1} \tilde{L} \hat{\otimes } \mhK[t]
		\end{align}
		and
		\begin{align}
		\label{W14B159}
		s \meqd \ubr{\gamma_1^{(p)}}_{\in \mmF_p \tilde{L}} - \hspace{-5pt} \ubr{\psi (x)}_{\minc{\eqref{W14B123}} \mmF_{p-r+1} \tilde{L}} \hspace{-15pt}.
		\end{align}
		From $\kappa_0 (t) + dt~ \kappa_1 (t)$ being a 1-cell in $\mmMC(\tilde{L})$, Lemma \ref{Lemma W14B1} indicates $\kappa_0 (t)$ to be the solution of the corresponding IVP.\\
		Integrating up to $t$ and $1$, using ${\kappa_0 (1)= U_\star (a_{p+1}) \hspace{-3pt}\minc{\eqref{W14B130}} \hspace{-3pt} \mmF_{p+1} \tilde{L}}$ and pulling out time-independent terms leads to
		\begin{align}
		\label{W14B160}
		\kappa_0 (t)=0 + t d_{\tilde{L}} (s)+ \int_0^t dt_1 \Big( \hspace{-8pt} \ubr{d_{\tilde{L}} (q(t_1))}_{\minc{\eqref{W14B158}} \mmF_{p+1} \tilde{L} \hat{\otimes } \mhK[t_1]} \hspace{-8pt} + \sum_{k=1}^{\infty} \frac{1}{k!} \{  \kappa_0 (t_1), \ldots, \kappa_0 (t_1), \kappa_1 (t_1)  \}_{k+1} \Big)
		\end{align}
		and
		\begin{align}
		\label{W14B161}
		\ubr{U_\star (a_{p+1})}_{\minc{\eqref{W14B130}} \mmF_{p+1} \tilde{L}} = 0 + d_{\tilde{L}} (s) + \int_{0}^{1} dt \Big( \hspace{-8pt} \ubr{d_{\tilde{L}} ( q(t))}_{\minc{\eqref{W14B158}} \mmF_{p+1} \tilde{L} \hat{\otimes } \mhK[t]} \hspace{-8pt} + \sum_{k=1}^\infty \frac{1}{k!} \{  \kappa_0 (t), \ldots, \kappa_0 (t), \kappa_1 (t)  \}_{k+1} \Big) .
		\end{align}
		Plugging $\kappa_1 (t) \in \mmF_{p-r+1} \tilde{L} \hat{\otimes } \mhK[t]$ into Equation \eqref{W14B161} directly yields
		\begin{align*}
		d_{\tilde{L}} (s) \in \mmF_{p-r+2} \tilde{L}.
		\end{align*}
		In turn, we can insert this into Equation \eqref{W14B160} and find
		\begin{align*}
		\kappa_0 (t) \in \mmF_{p-r+2} \tilde{L} \hat{\otimes } \mhK[t].
		\end{align*}
		This can then again be used in Equation \eqref{W14B161}, and so on and so forth.\\
		Eventually, this procedure leads to
		\begin{align}
		\label{W14B164}
		d_{\tilde{L}} (s) \in \mmF_{p+1} \tilde{L}
		\end{align}
		and
		\begin{align}
		\label{W14B165}
		\kappa_0 (t) \in \mmF_{p+1} \tilde{L} \hat{\otimes } \mhK[t].
		\end{align}
		Following the same line as in the proof of Lemma \ref{Lemma W14B2}, it is clear that
		\begin{align}
		\label{W14B166}
		\mu \meqd (\kappa_0 (t) + dt~ \kappa_1 (t)) + d_{\tilde{L}}^{\kappa_0 (t) + dt~ \kappa_1 (t)} ( dt ~\tilde{y}),
		\end{align}
		with $\tilde{y}$ as in Equation \eqref{W14B123}, defines a 1-cell in $\mmMC(\tilde{L})$ connecting $0$ with $U_\star (a_{p+1})$, too.\\
		Unravelling the notation and using $dt^2=0$, Equation \eqref{W14B166} amounts to (cf. Equation \eqref{W14B30})
		\begin{align*}
		\mu = \kappa_0 (t) + \Bigg( \kappa_1 (t) + d_{\tilde{L}} (\tilde{y}) + \ubr{  \{\hspace{-12pt}  \ubr{ \kappa_0 (t)}_{\minc{\eqref{W14B165}} \mmF_{p+1} \tilde{L} \hat{\otimes} \mhK[t]} , \ubr{ \tilde{y}}_{\minc{\eqref{W14B123}} \mmF_{p-r+1} \tilde{L}} \hspace{-10pt} \}_2 }_{\in \mmF_{p+1} \tilde{L} \hat{\otimes } \mhK[t]} + \text{higher filtration degrees} \Bigg) dt.
		\end{align*}
		By writing out $\kappa_1 (t)$ in the sense of Equation \eqref{W19A3}, we see that
		\begin{align*}
		\kappa_1 (t) + d_{\tilde{L}} (\tilde{y}) \meqc{\eqref{W19A3}} \ubr{\gamma_1^{(p)}- \psi (x) + d_{\tilde{L}} (\tilde{y}) }_{\meqc{\eqref{W14B125}} \mmF_{p+1} \tilde{L}} + \mcO( \mmF_{p+1} \tilde{L} \hat{\otimes} \mhK[t]) \in \mmF_{p+1} \tilde{L} \hat{\otimes } \mhK[t]
		\end{align*}
		and so it becomes immediate that the $dt$ coefficient of the 1-cell $\mu$ lies in $\mmF_{p+1} \tilde{L} \hat{\otimes} \mhK[t]$.\\
		Applying Lemma B2 from \cite{DolgushevLinfty} yields that there exists a rectified 1-cell $\gamma^{(p+1)}= \gamma_0^{(p+1)} (t) + dt~\gamma_1^{(p+1)}$ in $\mmMC(\tilde{L})$, which connects $0$ with $U_\star (a_{p+1})$ and satisfies $\gamma_1^{(p+1)} \in \mmF_{p+1} \tilde{L}$.\\
	\end{proof}
	\section{Higher Homotopy Groups}
	\label{Higher}
	\begin{theorem}[Isomorphism on higher Homotopy Groups]
		\label{Theorem W15D1}\hfill\\
		Let $L$ and $\tilde{L}$ be two $\mmS L_\infty$ algebras endowed with descending, bounded above and complete filtrations
		\begin{align*}
		\begin{aligned}
		L= \mmF_1 L \supset \mmF_2 L \supset \mmF_3 L \supset \ldots\\
		\tilde{L} = \mmF_1 \tilde{L} \supset \mmF_2 \tilde{L} \supset \mmF_3 \tilde{L} \supset \ldots
		\end{aligned}
		\end{align*}
		compatible with the $\mmS L_\infty$ algebra structures.\\
		Let $U:L \rightarrow \tilde{L}$ be an $\infty$-morphism of $\mmS L_\infty$ algebras compatible with the filtrations.\\
		Let $\psi$ be the linear term of $U$ as in Equation \eqref{W14B7} and let $U_\star$ be defined as in Equation \eqref{W14B8}.\\
		Let $\psi$ be a quasi-isomorphism on the r-1st page of the spectral sequences of the filtered complexes $(L,d_L)$ and $(\tilde{L},d_{\tilde{L}})$.\\
		Let us assume $H^0 ((\mmF_{2^q} L)/(\mmF_{\mmin(2^{q+1},r)} L))=0$ for every $q$ with $2^q < r$.\\
		Then for every Maurer-Cartan element $\tau \in \mMC(L)$ and for every $n \geq 1$, $\mmMC(U)$ induces a group isomorphism
		\begin{align}
		\label{W15D1}
		\pi_n (\mmMC(L), \tau) \stackrel{\cong}{\rightarrow} \pi_n (\mmMC(\tilde{L}), U_\star (\tau)).
		\end{align}
	\end{theorem}
	As it turns out, by virtue of Theorem 5.5 from \cite{Berglund}, this is an immediate consequence of the following lemma.
	\begin{lemma}
		\label{Lemma W15D1}\hfill\\	
		Let the situation be as in Theorem \ref{Theorem W15D1}.\\
		Let $\tau \in \mMC(L)$ be a Maurer-Cartan element such that $\tau \in \mmF_r L$.\\
		Then the linear part ${\psi^\tau : L^\tau \rightarrow \tilde{L}^{U_\star (\tau)}}$ of the twisted $\infty$-morphism ${U^\tau :L^\tau \rightarrow \tilde{L}^{U_\star (\tau)}}$ induces an isomorphism on the level of co-homology
		\begin{align}
		H(\psi^\tau) : H(L^\tau) \stackrel{\cong}{\rightarrow} H(\tilde{L}^{U_\star (\tau)}).
		\end{align}
	\end{lemma}
	\begin{proof}[\textbf{Proof of Lemma \ref{Lemma W15D1}}]
		\hfill\\
		We start by showing that the linear part of the twisted morphism also is a quasi-isomorphism on the r-1st page of the spectral sequences.\\
		As we have seen, it suffices to prove that also for $\psi$ replaced by $\psi^\tau$, Equation \eqref{W14B14a} still implies Equation \eqref{W14B14}.\\
		The composition $U^{\prime \tau}:S^+(L^\tau) \rightarrow \tilde{L}^{U_\star (\tau)}$ of the twisted morphism $U^\tau: L^\tau \rightarrow \tilde{L}^{U_\star (\tau)}$ (cf. Equation \eqref{W15C2}) with the projection (cf. Equation \eqref{W14B5}) is given by
		\begin{align*}
		U^{\prime \tau} (v_1, \ldots, v_n) = \sum_{k=0}^{\infty} \frac{1}{k!} U^\prime (\ubr{\tau, \ldots, \tau}_{\text{k-times}}, v_1, \ldots, v_n).
		\end{align*}
		The linear part $\psi^\tau$ is defined, analogously as in Equation \eqref{W14B7}, as the restriction $\psi^\tau \meqd U^{\prime \tau} \vert_{L^\tau}$ and so we find
		\begin{align}
		\label{W16J2}
		\psi^\tau (a) = \psi (a) + \sum_{k=1}^{\infty} \frac{1}{k!} U^\prime (\ubr{\tau, \ldots, \tau}_{\text{k-times}}, a).
		\end{align}
		Due to Equation \eqref{W15C1}, the twisted differential $d_L^\tau$, i.e. the differential on $L^\tau$, can be rewritten as
		\begin{align*}
		d_L^\tau (a) = d_L (a) + \sum_{k=1}^{\infty} \frac{1}{k!} \{\ubr{ \tau, \ldots, \tau}_{\text{k-times}},a  \}_{k+1},
		\end{align*}
		and the similar holds for $d_{\tilde{L}}^{U_\star (\tau)}$, too.\\
		We also notice that from having the same underlying filtered and graded vector spaces, the filtrations are the same on the twisted and untwisted $\mmS L_\infty$ algebras, hence $a \in \mmF_p L$ is equivalent to $a \in \mmF_p L^\tau$.\\
		Due to Equation \eqref{W16J2}, it is immediate, that for $q \in \mmF_p L^\tau$ 
		\begin{align*}
		\psi^\tau (q) = \psi (q)~\mmod~\mmF_{p+r} \tilde{L}
		\end{align*}
		holds.\\
		Similarly, $s \in \mmF_p L^\tau$ and $t \in \mmF_p \tilde{L}^{U_\star (\tau)}$ imply
		\begin{align*}
		d_{L}^\tau (s) = d_L (s) ~\mmod~\mmF_{p+r} L
		\end{align*}
		and
		\begin{align*}
		d_{\tilde{L}}^{U_\star (\tau)} (t) = d_{\tilde{L}} (t)~\mmod~\mmF_{p+r} \tilde{L},
		\end{align*}
		respectively.\\
		As a result of that, we can safely replace all the expressions in Equations \eqref{W14B14a} and \eqref{W14B14} with its twisted analogues and hence $\psi^\tau$ is a quasi-isomorphism on the r-1st page, too.\\
		Moreover, according to the Eilenberg-Moore Comparison Theorem (see Theorem 5.5.11 from \cite{Weibel}), $\psi^\tau$ induces an isomorphism
		\begin{align}
		\label{W16J11}
		H( \psi^\tau ): H(L^\tau) \stackrel{\cong}{\rightarrow} H(\tilde{L}^{U_\star (\tau)})
		\end{align}
		in co-homology.
	\end{proof}
	\begin{proof}[\textbf{Proof of Theorem \ref{Theorem W15D1}}]
		\hfill\\
		According to Lemma \ref{Lemma W19A1}, for every Maurer-Cartan element $\tau \in \mMC(L)$ there exists a Maurer-Cartan element $\tilde{\tau} \in \mMC(L)$, such that $\tilde{\tau} \in \mmF_r L$ and $\tau$ is gauge-equivalent to $\tilde{\tau}$.\\
		In the language of homotopy, $\tau$ and $\tilde{\tau}$ being gauge equivalent (with equivalence induced by 1-cells in $\mmMC(L)$) means $\tau$ and $\tilde{\tau}$ lie in the same connected component.\\
		Therefore, we can w.l.o.g. assume $\tau \in \mmF_r L$.\\
		According to Theorem 5.5 from \cite{Berglund}, for every $n \geq 1$ there is a group isomorphism
		\begin{align}
		\label{W16J12}
		B_{n-1}^\tau : H^{-n+1}(L^\tau) \stackrel{\cong}{\rightarrow} \pi_n (\mmMC(L),\tau).
		\end{align}
		Here, we use superscripts for indicating that we work in co-homology.\\
		In addition, the following diagram commutes:
		\[
		\begin{tikzcd}
		H^{-n+1} (L^\tau) \arrow[rr, "B_{n-1}^\tau"] \arrow[dd, "H^{-n+1} (\psi^\tau)"] &  & {\pi_n (\mmMC(L),\tau)} \arrow[dd, "{\mmMC(U)}"] \\
		&  &                                                       \\
		H^{-n+1} (\tilde{L}^{U_\star (\tau)}) \arrow[rr, "B_{n-1}^{U_\star (\tau)}"]    &  & {\pi_n (\mmMC(\tilde{L}), U_\star (\tau))}           
		\end{tikzcd}
		\]
		As a result of that, we are allowed to apply Lemma \ref{Lemma W15D1} and deduce the vertical left arrow to be a group isomorphism. Additionally, from the fact that in the above diagram both horizontal arrows are indeed group isomorphisms, clearly the right vertical arrow $\mmMC(U)$ is a group isomorphism as well.
	\end{proof}
	\newpage
	\appendix
	\section{Proof of Theorem \ref{Theorem ZA1}}
	\label{Zusatzfrage}
	In the verification of Theorem \ref{Theorem ZA1} we use similar steps as in Sections \ref{Bijection} and \ref{Higher}. As before, we present the proof in the setting of $\mmS L_\infty$ algebras.\\
	\begin{theorem}[Theorem \ref{Theorem ZA1}]
		\label{Theorem ZA2}\hfill\\
		Let $L$ and $\tilde{L}$ be two $\mmS L_\infty$ algebras equipped with descending, bounded above and complete filtrations
		\begin{align*}
		\begin{aligned}
		L= \mmF_0 L \supset \mmF_1 L \supset \mmF_2 L \supset \ldots\\
		\tilde{L}= \mmF_1 \tilde{L} \supset \mmF_2 \tilde{L} \supset \mmF_3 \tilde{L} \supset \ldots
		\end{aligned}
		\end{align*}
		compatible with the $\mmS L_\infty$ structures. In addition, let $L$ be Abelian.\\
		Let $U:L \rightarrow \tilde{L}$ be an $\infty$-morphism of $\mmS L_\infty$ algebras compatible with the filtrations that has its linear part, say $\psi$, being a quasi-isomorphism on the r-1st page of the spectral sequences induced by the filtered complexes $(L,d_L)$ and $(\tilde{L}, d_{\tilde{L}})$ for some arbitrary fixed $r$. Moreover, let us assume $U$ to raise the degree of filtration by $r-1$\footnote{See Equation \eqref{Raise}} and $U_\star$ to be a finite sum\footnote{See Equation \eqref{W14B8}. This is equivalent to demanding that for every $a \in L$ there exists a $N \in \mathbb{N}$ such that for $U^\prime (\ubr{a, \ldots, a}_{\text{n-times}})=0$ for all $n \geq N$.}.\\
		Let also $H^0((\mmF_{2^q} \tilde{L})/(\mmF_{\mmin (2^{q+1},r)} \tilde{L}))=0$ hold for every $q$ with $2^q < r$.\\
		Then
		\begin{align}
		U_\star: \mMC(L)/\sim \stackrel{\cong}{\rightarrow} \mMC(\tilde{L}) /\sim
		\end{align}
		is a bijection.
	\end{theorem}
	\begin{proof}[Proof of Theorem \ref{Theorem ZA2}]
		\hfill\\	
		We show surjectivity and injectivity separately.
		\subsection{Surjectivity}\hfill\\
		We prove by induction on $p$ the following statement, from which surjectivy directly follows:\\
		\underline{Statement:}\\
		Let $b\in \mMC(\tilde{L})$ be arbitrary.\\
		Then there exists a sequence $\{ a_p \}_{p \geq r+1}$ of degree $0$ elements in $L$, a sequence $\{b_p\}_{p \geq r+1}$ of Maurer-Cartan elements in $\mMC(\tilde{L})$ and a sequence $\{ \gamma^{(p)} \}_{p \geq r+2}$ of rectified 1-cells in $\mmMC(\tilde{L})$, such that
		\begin{enumerate}
			\item $b_{r+1} \sim b$.
			\item $a_p \in \mmF_1 L$ and $a_{p} - a_{p-1} \in \mmF_{p-r} L$.
			\item $\mcurv(a_p) \in \mmF_p L$.
			\item $ b_p \in \mmF_r \tilde{L}$ and $\gamma^{(p)} = \gamma_0^{(p)} (t) + dt~\gamma_1^{(p)}$ satisfies $\gamma_0^{(p)}(0) = b_{p-1}$ and $\gamma_0^{(p)} (1) = b_p$ (thus $b_{p} \sim b_{p-1}$), as well as $\gamma_1^{(p)} \in \mmF_{p-r} \tilde{L}$.\\
			In addition, $b_p- b_{p-1} \in \mmF_{p-1} \tilde{L}$ holds.
			\item $U_\star (a_p) = b_p~\mmod ~\mmF_p \tilde{L}$.
		\end{enumerate}
		Notice the adjustment made in Point 2 when compared to the statement from Section \ref{Bijection}.\\
		\underline{Proof of Statement:}\hfill\\
		\underline{Base of induction:}\hfill\\
		\underline{Constructing $a_{r+1}$:}\hfill\\
		Because of Lemma \ref{Lemma W19A1}, there exists a Maurer-Cartan element $b_r \in \mMC(\tilde{L})$ of filtration degree $b_r \in \mmF_r \tilde{L}$, which is gauge equivalent to $b$.\\
		From $b_r$ being a Maurer-Cartan element it particularly satisfies $d_{\tilde{L}} (b_r)\in \mmF_{r+r} \tilde{L}$. Hence, all the requirements of Equation \eqref{W14B14a} in the case of $p=r$, $b_r$ in the role of $b$ and $a$ set to zero, are satisfied.\\
		Thus, Equation \eqref{W14B14} yields:
		\begin{align}
		\label{W20B1}
		&\left\{
		\begin{aligned}
		\exists a_{r+1} &\in \mmF_1 L\\
		\exists y &\in \mmF_1 \tilde{L}
		\end{aligned}
		\right.\\
		\nonu \mst\\
		\label{W20B2a}
		&\left\{
		\begin{aligned}
		d_L (a_{r+1}) &\in \mmF_r L\\
		\psi (a_{r+1} ) - d_{\tilde{L}} (y) &\in \mmF_r \tilde{L}
		\end{aligned}
		\right.\\	
		\nonumber \mand\\
		\label{W20B2}
		&\left\{
		\begin{aligned}
		d_L (a_{r+1}) &\in \mmF_{r+1}	L\\
		b_{r} - \psi (a_{r+1}) + d_{\tilde{L}} (y) &\in \mmF_{r+1} \tilde{L}.
		\end{aligned}
		\right.	
		\end{align}
		But then, using the definition of $\mcurv$, a simple computation reveals
		\begin{align*}
		\mcurv(a_{r+1}) = \ubr{d_L (a_{r+1})}_{\minc{\eqref{W20B2}} \mmF_{r+1} L} + \ubr{\sum_{m=2}^{\infty} \frac{1}{m!} \{a_{r+1}, \ldots,a_{r+1} \}_m}_{= 0},
		\end{align*}
		i.e.
		\begin{align}
		\label{W20B3}
		\mcurv (a_{r+1}) \in \mmF_{r+1} L.
		\end{align}
		\underline{Constructing $b_{r+1}$:}\hfill\\
		Since $U$ raises the degree of filtration by $r-1$, we know
		\begin{align}
		\label{W20B3b}
		\psi (a_{r+1}) \in \mmF_{r} \tilde{L}.
		\end{align}
		Together with Equation \eqref{W20B2a} this leads to
		\begin{align}
		\label{W20B5}
		d_{\tilde{L}} (y) \in \mmF_r \tilde{L}.
		\end{align}
		We continue by constructing a rectified 1-cell $\gamma = \gamma_0 (t) + dt~\gamma_1$ in $\mmMC(\tilde{L})$ in the usual manner by setting:
		\begin{align}
		\label{W20B9}
		\begin{aligned}
		\gamma_0 (0) &\meqd b_r\\
		\gamma_1 &\meqd y
		\end{aligned}	
		\end{align}
		According to Lemma \ref{Lemma W14B1}, $\gamma_0 (t)$ is found to be
		\begin{align*}
		\gamma_0 (t) = \ubr{b_r}_{\in \mmF_r \tilde{L}} + \int_{0}^{t} dt_1 \Big( \ubr{d_{\tilde{L}} (y)}_{\minc{\eqref{W20B5}} \mmF_r \tilde{L}} + \sum_{k=1}^{\infty} \frac{1}{k!} \{ \gamma_0 (t_1), \ldots, \gamma_0 (t_1), y \}_{k+1} \Big),
		\end{align*}
		so particularly
		\begin{align}
		\label{W20B10}
		\gamma_0 (t) \in \mmF_r \tilde{L} \hat{\otimes} \mhK[t]
		\end{align}
		holds.\\
		We set $b_{r+1}$ to be the endpoint of this 1-cell, i.e.
		\begin{align}
		\label{W20B11}
		b_{r+1} \meqd \gamma_0 (1) \minc{\eqref{W20B10}} \mmF_r \tilde{L}
		\end{align}
		and so we find
		\begin{align}
		\label{W20B12}
		\begin{aligned}
		b_{r+1} &= b_r + \int_{0}^{1} dt \Big( d_{\tilde{L}} (y) + \ubr{\sum_{k=1}^{\infty} \frac{1}{k!} \{ \hspace{-9pt} \ubr{\gamma_0 (t)}_{\minc{\eqref{W20B10}}\mmF_r \tilde{L} \hat{\otimes} \mhK[t]} \hspace{-9pt}, \ldots, \gamma_0 (t), y \}_{k+1}}_{\in \mmF_{r+1} \tilde{L} \hat{\otimes} \mhK[t]} \Big)\\[-6pt]
		&=b_r + d_{\tilde{L}} (y) + \mcO(\mmF_{r+1} \tilde{L}).
		\end{aligned}
		\end{align}
		\underline{Proving $U_\star (a_{r+1})=b_{r+1}~\mmod~ \mmF_{r+1} \tilde{L}$:}\hfill\\
		From the fact that $U$ raises the degree of filtration by $r-1$, we recognise
		\begin{align}
		\label{W20B13}
		U_\star (a_{r+1}) = \psi (a_{r+1}) + \mcO (\mmF_{r+1} \tilde{L}).
		\end{align}
		But then, using both Equations \eqref{W20B12} and \eqref{W20B13}, we find
		\begin{align}
		\label{W20B14}
		\begin{aligned}
		U_\star (a_{r+1}) - b_{r+1} &\meqc{\substack{\eqref{W20B12}\\\eqref{W20B13}}} \Big( \psi (a_{r+1}) + \mcO (\mmF_{r+1} \tilde{L}) \Big) - \Big( b_r + d_{\tilde{L}} (y) + \mcO(\mmF_{r+1} \tilde{L} ) \Big)\\[-0.1cm]
		&\hspace{8pt}=-\Big( \ubr{b_r - \psi (a_{r+1})+ d_{\tilde{L}} (y)}_{\minc{\eqref{W20B2}}  \mmF_{r+1} \tilde{L}} \Big) + \mcO(\mmF_{r+1} \tilde{L}) \in \mmF_{r+1} \tilde{L}.
		\end{aligned}
		\end{align}\hfill\\[-0.8cm]
		\underline{Induction Step:}\hfill\\
		\underline{Preparation:}\hfill\\
		We start by setting
		\begin{align}
		\label{W20B15}
		\tilde{b} \meqd U_\star (a_p) - b_p \in \mmF_p \tilde{L},
		\end{align}
		which holds due to assumption.\\
		Next, we know
		\begin{align*}
		d_L (\mcurv(a_p)) =0
		\end{align*}\hfill\\[-0.4cm]
		from $L$ being Abelian, thus
		\begin{align}
		\label{W20B16}
		d_L (\mcurv (a_p)) \in \mmF_{p+r} L
		\end{align}
		holds, trivially.\\
		Furthermore, we may compute
		\begin{align*}
		\begin{aligned}
		\psi (\mcurv (a_p))  &\meqc{\text{(2.18) from \cite{DolgushevEnhancement}}} \mcurv (U_\star (a_p)) - \ubr{\sum_{m=1}^{\infty} \frac{1}{m!} U^\prime (\ubr{a_p, \ldots, a_p}_{\text{m times}}, \ubr{\mcurv(a_p)}_{\in \mmF_p L})}_{\in \mmF_{p+r}\tilde{L}} \\
		&\hspace{20pt}=\mcurv (U_\star (a_p) ) + \mcO(\mmF_{p+r} \tilde{L}) \meqc{\eqref{W20B15} } \mcurv (\ubr{\tilde{b}}_{\minc{\eqref{W20B15}} \mmF_p \tilde{L}} + b_p) + \mcO(\mmF_{p+r} \tilde{L})\\[-0.1cm]
		&\meqc{\text{(2.20) from \cite{DolgushevEnhancement}}}\ubr{\mcurv(b_p)}_{\meqc{b_p \in \mMC(\tilde{L})} 0} + d_{\tilde{L}}^{b_p} (\tilde{b}) + \sum_{m=2}^{\infty} \frac{1}{m!} \{ \tilde{b}, \ldots, \tilde{b}\}_m^{b_p} + \mcO(\mmF_{p+r} \tilde{L}) \\
		&\hspace{20pt}=d_{\tilde{L}} (\tilde{b}) + \ubr{\sum_{k=1}^{\infty} \frac{1}{k!} \{ \ubr{b_p}_{\in \mmF_r \tilde{L}} , \ldots, b_p, \ubr{\tilde{b}}_{\minc{ \eqref{W20B15}} \mmF_p \tilde{L}}  \}_{k+1}}_{\in \mmF_{p+r} \tilde{L}}\\[-0.2cm]
		&\hspace{20pt}+\ubr{\sum_{m=2}^{\infty} \frac{1}{m!} \sum_{k=0}^{\infty} \frac{1}{k!} \{\ubr{b_p, \ldots, b_p}_{\text{k-times}} , \ubr{\tilde{b}}_{\minc{\eqref{W20B15}} \mmF_p \tilde{L}} , \ldots, \tilde{b} \}_{k+m}}_{\in \mmF_{2p} \tilde{L} \stackrel{p \geq r+1}{\subset} \mmF_{p+r} \tilde{L}} + \mcO(\mmF_{p+r} \tilde{L}),
		\end{aligned}	
		\end{align*}
		i.e. 
		\begin{align}
		\label{W20B17}
		\psi (\mcurv (a_p)) - d_{\tilde{L}} (\tilde{b}) \in \mmF_{p+r} \tilde{L}
		\end{align}
		holds.\\
		Due to Equations \eqref{W20B16} and \eqref{W20B17}, all the requirements of Equation \eqref{W14B14a}, in the setting of $p=p$, $\mcurv(a_p)$ in the role of $a$ and $\tilde{b}$ playing the role of $b$, are satisfied.\\
		Thus, Equation \eqref{W14B14} implies
		\begin{align}
		\label{W20B18}
		&\left\{
		\begin{aligned}
		\exists \tilde{a} &\in \mmF_{p-r+1} L\\
		\exists y &\in \mmF_{p-r+1} \tilde{L}
		\end{aligned}
		\right.\\
		\nonu \mst\\
		\label{W20B19}
		&\left\{
		\begin{aligned}
		d_L (\tilde{a}) &\in \mmF_p L\\
		\psi (\tilde{a} ) - d_{\tilde{L}} (y) &\in \mmF_p \tilde{L}
		\end{aligned}
		\right.\\	
		\nonumber \mand\\
		\label{W20B20}
		&\left\{
		\begin{aligned}
		\mcurv( a_p) - d_L (\tilde{a}) &\in \mmF_{p+1}	L\\
		\tilde{b} - \psi (\tilde{a}) + d_{\tilde{L}} (y) &\in \mmF_{p+1} \tilde{L}.
		\end{aligned}
		\right.	
		\end{align}
		\underline{Constructing $a_{p+1}$:}\hfill\\
		Differently than in Section \ref{Bijection} we do not need to distinguish the two cases but can rather use the same approach for both.\\
		We set
		\begin{align}
		\label{W20B21}
		a_{p+1} \meqd a_p - \tilde{a}.
		\end{align}
		A small calculation shows
		\begin{align*}
		\begin{aligned}
		&\mcurv (a_{p+1})  \meqc{\eqref{W20B21}} \mcurv (a_p - \tilde{a}) \meqc{\text{(2.20) from \cite{DolgushevEnhancement}}} \mcurv(a_p) + d_L^{a_p} (- \tilde{a}) + \sum_{m=2}^{\infty} \frac{1}{m!} \{ -\tilde{a}, \ldots, - \tilde{a}  \}_m^{a_p} \\
		&=\ubr{\mcurv(a_p) - d_L (\tilde{a})}_{\minc{\eqref{W20B20}} \mmF_{p+1} L} + \ubr{ \sum_{k=1}^{\infty} \frac{1}{k!} \{a_p, \ldots, a_p , - \tilde{a}\}_{k+1} }_{= 0} + \ubr{\sum_{m=2}^{\infty} \frac{1}{m!} \sum_{k=0}^{\infty} \frac{1}{k!} \{a_p, \ldots, a_p, - \tilde{a}, \ldots, - \tilde{a}\}_{k+m}}_{= 0},
		\end{aligned}
		\end{align*}
		i.e.
		\begin{align}
		\label{W20B22}	
		\mcurv(a_{p+1}) \in \mmF_{p+1} L.
		\end{align}
		\underline{Constructing $b_{p+1}$:}\hfill\\
		Using Equation \eqref{W20B18} and the fact that $U$ raises the degree of filtration by $r-1$ yields
		\begin{align*}
		\psi ( \tilde{a}) \in \mmF_{p} \tilde{L}.
		\end{align*}
		Plugging this into Equation \eqref{W20B19} results in
		\begin{align}
		\label{W20B29}
		d_{\tilde{L}} (y) \in \mmF_p \tilde{L}.
		\end{align}
		Let $\gamma^{(p+1)} = \gamma_0^{(p+1)} (t) + dt~\gamma_1^{(p+1)}$ be the rectified 1-cell in $\mmMC(\tilde{L})$ determined by setting
		\begin{align*}
		\begin{aligned}
		\gamma_1^{(p+1)} \meqd -y\\
		\gamma_0^{(p+1)} (0) \meqd b_p.
		\end{aligned}
		\end{align*}
		According to Lemma \ref{Lemma W14B1}, $\gamma_0^{(p+1)} (t)$ is then given by
		\begin{align}
		\label{W20B30}
		\gamma_0^{(p+1)} (t) =\hspace{-4pt} \ubr{b_p}_{\in \mmF_r \tilde{L}} \hspace{-5pt} + \hspace{-2pt} \int_{0}^{t} \hspace{-4pt} dt_1 \Big( \hspace{-2pt} \ubr{d_{\tilde{L}} (-y)}_{\minc{\eqref{W20B29}} \mmF_p \tilde{L}}  + \sum_{m=1}^{\infty} \frac{1}{m!} \{ \gamma_0^{(p+1)} (t_1), \ldots, \gamma_0^{(p+1)} (t_1), -y  \}_{m+1} \Big).
		\end{align}
		Hence, we may deduce in the usual manner
		\begin{align}
		\label{W20B30b}
		\gamma_0^{(p+1)} (t) \in \mmF_r \tilde{L} \hat{\otimes} \mhK[t].
		\end{align}
		Next, we set $b_{p+1}$ to be the endpoint of $\gamma^{(p+1)}$.\\
		From Equation \eqref{W20B30b} it is immediate that
		\begin{align}
		\label{W20B30c}
		b_{p+1} \in \mmF_r \tilde{L}.
		\end{align}
		Furthermore, a short calculation also shows
		\begin{align*}
		b_{p+1} \meqd \gamma_0^{(p+1)} (1) = b_p - d_{\tilde{L}} (y) + \ubr{\int_{0}^{1} dt ~ \sum_{m=1}^{\infty} \frac{1}{m!} \{ \ubr{\gamma_0^{(p+1)} (t)}_{\minc{\eqref{W20B30b}} \mmF_r \tilde{L}\hat{\otimes} \mhK[t]} , \ldots, \gamma_0^{(p+1)} (t),  \hspace{-10pt}\ubr{-y}_{\minc{\eqref{W20B18}} \mmF_{p-r+1} \tilde{L}} \hspace{-5pt}   \}_{m+1}}_{\in \mmF_{p+1} \tilde{L}} ,
		\end{align*}
		i.e.
		\begin{align}
		\label{W20B26}
		b_{p+1}= b_p - d_{\tilde{L}} (y) + \mcO (\mmF_{p+1} \tilde{L}).
		\end{align}
		\underline{Proving $U_\star (a_{p+1})=b_{p+1}~\mmod ~\mmF_{p+1} \tilde{L}$:}\\
		The difference between $U_\star (a_{p+1})$ and $b_{p+1}$ is found to be
		\begin{align*}
		\begin{aligned}
		&U_\star (a_{p+1}) - b_{p+1} \\
		&\hspace{-8pt}\meqc{\substack{\eqref{W20B21} \\ \eqref{W20B26}}}\bigg( \psi(a_p) - \psi (\tilde{a}) + \sum_{m=2}^{\infty} \frac{1}{m!} U^\prime \Big( (a_p - \tilde{a}), \ldots, (a_p - \tilde{a}) \Big) \bigg) - \Big( b_p - d_{\tilde{L}} (y) + \mcO(\mmF_{p+1} \tilde{L}) \Big)\\
		&=\ubr{- \psi(\tilde{a}) + \ubr{U_\star (a_p) - b_p}_{\meqc{\eqref{W20B15}} \tilde{b}} + d_{\tilde{L}} (y)}_{\minc{\eqref{W20B20}} \mmF_{p+1} \tilde{L}} + \bigg( \sum_{m=2}^{\infty} \frac{1}{m!} U^\prime \Big( (a_p- \tilde{a}), \ldots, (a_p - \tilde{a}) \Big) + \psi (a_p)- U_\star (a_p) \bigg) + \mcO (\mmF_{p+1} \tilde{L})\\
		&=\sum_{m=2}^{\infty} \frac{1}{m!} U^\prime \Big( (a_p- \tilde{a}), \ldots, (a_p - \tilde{a}) \Big) - \sum_{m=2}^{\infty} \frac{1}{m!} U^\prime \Big( a_p, \ldots, a_p \Big)+ \mcO (\mmF_{p+1} \tilde{L} ).
		\end{aligned}
		\end{align*}
		It remains to verify that the difference of the two series has filtration degree $\mmF_{p+1} \tilde{L}$. This can be done order by order.\\
		For e.g. $m=2$ the approach is:
		\begin{align*}
		\begin{aligned}
		&U^\prime \big( (a_p - \tilde{a}) , (a_p - \tilde{a}) \big) - U^\prime (a_p, a_p) \\
		&=U^\prime \big( (a_p- \tilde{a}), (a_p - \tilde{a}) \big) - U^\prime (a_p - \tilde{a}+ \tilde{a} , a_p)\\
		&=U^\prime \big( (a_p- \tilde{a}), (a_p - \tilde{a}) \big) - U^\prime \big( (a_p - \tilde{a}) , a_p \big) - \ubr{U^\prime (\hspace{-7pt}\ubr{\tilde{a}}_{\minc{\eqref{W20B18}} \mmF_{p-r+1} L} \hspace{-7pt}, a_p )}_{\in \mmF_{p+1} \tilde{L}} \\
		&=U^\prime \big( (a_p- \tilde{a}), (a_p - \tilde{a}) \big)- U^\prime \big( (a_p - \tilde{a}), (a_p - \tilde{a} + \tilde{a}) \big) + \mcO (\mmF_{p+1} \tilde{L})\\
		&=\ubr{ U^\prime \big( (a_p- \tilde{a}), (a_p - \tilde{a}) \big)- U^\prime \big( (a_p- \tilde{a}), (a_p - \tilde{a}) \big)}_{=0} - \ubr{ U^\prime \big( (a_p - \tilde{a}), \tilde{a} \big) }_{\in \mmF_{p+1} \tilde{L}} + \mcO(\mmF_{p+1} \tilde{L} ) \in \mmF_{p+1} \tilde{L},
		\end{aligned}
		\end{align*}
		where in the third and the fifth line we used the fact that $U$ raises the degree of filtration by $r-1$.\\
		Hence,
		\begin{align}
		\label{W20B28}
		U_\star (a_{p+1}) = b_{p+1}~\mmod~\mmF_{p+1} \tilde{L}
		\end{align}
		holds, indeed.
		\subsection{Injectivity}\hfill\\
		By induction on $p$, we assert the following statement, from which injectivity directly follows.\\
		\underline{Statement:}\hfill\\
		Let $a\in \mMC(L)$ be an arbitrary Maurer-Cartan element which satisfies $U_\star (a) \sim 0$.\\
		Then there exists a sequence $\{a_p\}_{p \geq 1}$ of Maurer-Cartan elements in $\mMC(L)$ and sequences $\{\xi^{(p)}\}_{p \geq 2}$ and $\{\gamma^{(p)}\}_{p \geq 1}$ of rectified 1-cells in $\mmMC(L)$ and $\mmMC(\tilde{L})$, respectively, such that
		\begin{enumerate}
			\item $a_1 \sim a$.
			\item $a_p \in \mmF_p L$.
			\item $\xi^{(p)}= \xi_0^{(p)} (t)+ dt~\xi_1^{(p)}$ satisfies $\xi_0^{(p)} (0) = a_{p-1}$ and $\xi_0^{(p)} (1) = a_p$ (thus $a_p \sim a_{p-1}$), as well as $\xi_1^{(p)} \in \mmF_{p-r} L$.
			\item $\gamma^{(p)}= \gamma_0^{(p)} (t) + dt~\gamma_1^{(p)}$ satisfies $\gamma_0^{(p)} (0)=0$ and $\gamma_0^{(p)} (1)= U_\star (a_p)$ (thus $0 \sim U_\star (a_p)$).\\
			In addition, $\gamma_1^{(p)} \in \mmF_p \tilde{L}$ holds.
		\end{enumerate}\hfill\\
		Notice, that here the statement starts at $p=1$ rather than $p=r$ as in Section \ref{Bijection}.\\
		\underline{Proof of Statement:}\hfill\\
		\underline{Base of Induction:}\hfill\\
		By assumption we have $U_\star (a) \sim 0$, i.e. there exists a (due to Lemma B2 from \cite{DolgushevLinfty} w.l.o.g. rectified) 1-cell $\gamma= \gamma_0 (t) + dt~\gamma_1$ in $\mmMC(\tilde{L})$ connecting $U_\star (a)$ and $0$.\\
		According to Lemma \ref{Lemma W14B1}, this 1-cell is supposed to satisfy
		\begin{align}
		\label{DA1}
		\gamma_0 (t)= 0 + t d_{\tilde{L}} (\gamma_1) + \int_{0}^{t} dt_1~\sum_{m=1}^{\infty} \frac{1}{m!} \{ \gamma_0 (t_1), \ldots, \gamma_0 (t_1), \gamma_1 \}_{m+1}
		\end{align} 
		and
		\begin{align}
		\label{DA2}
		\ubr{U_\star (a)}_{\in \mmF_{r-1} \tilde{L}} = \gamma_0 (1) = 0 + d_{\tilde{L}} (\gamma_1) + \int_{0}^{1} dt~\sum_{m=1}^{\infty} \frac{1}{m!} \{ \gamma_0(t), \ldots, \gamma_0 (t), \gamma_1 \}_{m+1},
		\end{align}
		where we used the fact that $U$ raises the degree of filtration by $r-1$.\\
		In the usual manner, this implies
		\begin{align}
		\label{DA3}
			\gamma_0 (t) \in \mmF_{r-1} \tilde{L} \hat{\otimes } \mhK[t]
		\end{align}
		and
		\begin{align}
		\label{DA4}
			d_{\tilde{L}} (\gamma_1 ) \in \mmF_{r-1} \tilde{L}
		\end{align}
		to hold.\\
		Inserting Equation \eqref{DA3} into Equation \eqref{DA2} directly yields
		\begin{align}
		\label{DA5}
			U_\star (a) - d_{\tilde{L}} ( \gamma_1) = \int_{0}^{1} dt ~\sum_{m=1}^{\infty} \frac{1}{m!} \ubr{\{ \hspace{-10pt}\ubr{\gamma_0 (t)}_{\minc{\eqref{DA3}} \mmF_{r-1} \tilde{L} \hat{\otimes } \mhK[t]}\hspace{-10pt} , \ldots , \gamma_0 (t), \gamma_1  \}_{m+1}}_{\in \mmF_r  \tilde{L}\hat{\otimes }  \mhK[t]} \in \mmF_{r} \tilde{L}.
		\end{align}
		For an application of Equation \eqref{W14B14a} we need $\psi (a)$ rather than $U_\star (a)$, thus some additional work needs to be done.\\
		By definition, we have
		\begin{align}
		\label{DA6}
			\psi (a) \meqc{\eqref{W14B8}} U_\star (a) - \sum_{m=2}^{\infty} \frac{1}{m!} U^\prime (\ubr{a, \ldots, a}_{\text{m- times}}),
		\end{align}
		where the sum consists of only finitely many terms by assumption.\\
		Let us further investigate on the degree of filtration of the differential of the redundant terms. By using the definition of an $\infty$-morphisms of $\mmS L_\infty$ algebras (e.g. see Proposition 10.2.7. from \cite{Loday}) we know
		\begin{align}
		\label{DA7a}
		d_{\tilde{L}} \big( U^\prime (a, \ldots, a) \big)
		\end{align}
		to consist of terms of the form
		\begin{align}
		\label{DA7}
		\begin{aligned}
		U^\prime(a, \ldots, d_L(a), \ldots,a),~ U^\prime (a, \ldots,\{a, \ldots,a\},\ldots,a)~\mand\\
		\{U^\prime (a,\ldots,a), \ldots, U^\prime (a, \ldots,a)\}
		\end{aligned}
		\end{align}
		only.\\
		Let us explain why every of these three possible forms carries $\mmF_r \tilde{L}$ as the degree of filtration: First, $d_L (a)=0$ holds, due to $a$ being a Maurer-Cartan element on the Abelian $\mmS L_\infty$ algebra $L$. But also, $\{a,\ldots,a\}=0$ is obvious from $L$ being Abelian. Thus, neither of the first two terms has a non-vanishing contribution. Eventually from $U$ raising the degree of filtration by $r-1$ we clearly see
		\begin{align*}
			\ubr{\{\ubr{U^\prime (a, \ldots, a)}_{\in \mmF_{r-1} \tilde{L}} , \ldots ,\ubr{U^\prime (a, \ldots,a)}_{\in \mmF_{r-1} \tilde{L}} \}}_{\in \mmF_{2(r-1)}\tilde{L} \subset \mmF_r \tilde{L}}.
		\end{align*}
		Thus, according to this reasoning, it is safe to say that for every of the finitely many unwanted summands the degree of filtration of the differential lies in $\mmF_r \tilde{L}$. But also $U^\prime (a, \ldots,a) \in \mmF_{r-1} \tilde{L}$ holds, allowing us to make use of $H^0 ((\mmF_{2^t} \tilde{L})/(\mmF_r \tilde{L}) )=0$ for $t$ being the highest power of $2$ satisfying $2^t <r$.\\
		To be more precise, from $U^\prime (a,\ldots,a) \in \mmF_{r-1} \tilde{L} \subset \mmF_{2^t} \tilde{L}$ and $d_{\tilde{L}} (U^\prime (a, \ldots, a)) \in \mmF_r \tilde{L}$ we deduce 
		\begin{align}
		\label{DA8}
		\begin{aligned}
			\exists u &\in \mmF_{2^t} \tilde{L}~\mst\\
			U^\prime (a, \ldots , a) - d_{\tilde{L}} (u) &\in \mmF_r \tilde{L}.
		\end{aligned}
		\end{align}
		We do this for each of the (finitely many, by assumption) unwanted terms and thus make sure that for $u= \frac{1}{2!} u^{(2)} + \frac{1}{3!} u^{(3)} + \ldots$
		\begin{align}
		\label{DA9}
			\sum_{m=2}^{\infty} \frac{1}{m!} U^\prime (a, \ldots, a) -d_L (u) \in \mmF_r \tilde{L}
		\end{align}
		holds.\\
		From combining these results, we find
		\begin{align}
		\label{DA10}
		\begin{aligned}
			\psi (a) - d_{\tilde{L}} (\gamma_1 -u) \meqc{\eqref{W14B8}} \ubr{U_\star (a) -d_{\tilde{L}} (\gamma_1)}_{\minc{\eqref{DA5}} \mmF_r \tilde{L}} - \ubr{ \sum_{m=2}^{\infty} \frac{1}{m!} U^\prime (a, \ldots, a) - d_{\tilde{L}} (u)}_{\minc{\eqref{DA9}} \mmF_{r} \tilde{L}} \in \mmF_r \tilde{L}.			
		\end{aligned}	
		\end{align}
		For this reason, all the requirements of Equation \eqref{W14B14a} in the case of $p=0$, $a=a$ and $b$ set to $\gamma_1- u$, are satisfied.\\
		Thus, Equation \eqref{W14B14} implies
		\begin{align}
		\label{DA11}
		\begin{aligned}
			\exists x &\in \mmF_0 L~\mst\\
			a-d_L (x) &\in \mmF_1 L.
		\end{aligned}	
		\end{align}
		We may use this $x$ for constructing a 1-cell in $\mmMC(L)$. More precisely, let $\xi = \xi_0 (t) + dt~\xi_1$ be the rectified 1-cell in $\mmMC(L)$ defined by means of setting the starting point to $a$ and demanding
		\begin{align}
		\label{DA12}
			\xi_1 \meqd -x. 
		\end{align}
		Setting $a_1$ to be the endpoint of this 1-cell and integrating the corresponding (in the sense of Lemma \ref{Lemma W14B1}) IVP yields
		\begin{align}
		\label{DA13}
		\begin{aligned}
			a_1\meqd \xi_0 (1) &= a+ \int_{0}^{1} dt \Big( d_L (-x) + \ubr{\sum_{m=1}^{\infty} \frac{1}{m!} \{\xi_0 (t), \ldots, \xi _0 (t),-x\}_{m+1}}_{=0} \Big)\\[-8pt]
			&=\ubr{a- d_L (x)}_{\minc{\eqref{DA11}} \mmF_1 L} \in \mmF_1 L.
		\end{aligned}
		\end{align}
		\underline{Induction Step:}\hfill\\
		\underline{Construction of $a_{p+1}$:}\hfill\\
		We follow the same idea as in Section \ref{Bijection}, sometimes making minor adjustments due to the fact that $p \geq r$ can no longer be assumed.\\
		Because $a_p \in \mMC(L)$ and $L$ is Abelian, we find
		\begin{align*}
		d_L (a_p) + \ubr{\sum_{k=2}^{\infty} \frac{1}{k!} \{ a_p, \ldots, a_p \}_k }_{=0} =0,
		\end{align*}
		i.e.
		\begin{align}
		\label{W20B75}
		d_L (a_p) \in \mmF_{p+r} L
		\end{align}
		holds, trivially.\\
		From $\gamma^{(p)}$ connecting $0$ and $U_\star (a_p)$, $\gamma_0^{(p)}$ must be the solution of the corresponding IVP according to Lemma \ref{Lemma W14B1}. Thus, integration up to $t$ and $1$, respectively, leads to
		\begin{align}
		\label{W20B71}
		\gamma_0^{(p)} (t)= 0 + t d_{\tilde{L}} (\gamma_1^{(p)}) + \int_{0}^{t} dt_1~\sum_{m=1}^{\infty} \frac{1}{m!} \{ \gamma_0^{(p)} (t_1), \ldots, \gamma_0^{(p)} (t_1), \gamma_1^{(p)} \}_{m+1}
		\end{align} 
		and
		\begin{align}
		\label{W20B72}
		\ubr{U_\star (a_p)}_{\in \mmF_{p+(r-1)} \tilde{L}} \hspace{-5pt}= \gamma_0^{(p)} (1) = 0 + d_{\tilde{L}} (\gamma_1^{(p)}) + \int_{0}^{1} dt~\sum_{m=1}^{\infty} \frac{1}{m!} \{ \gamma_0^{(p)}(t), \ldots, \gamma_0^{(p)} (t), \gamma_1^{(p)} \}_{m+1},
		\end{align}
		where we used the facts that $a_p \in \mmF_p L$ and $U$ raises the degree of filtration by $r-1$.\\
		From this we may deduce in the usual manner
		\begin{align}
		\label{W20B73}
		d_{\tilde{L}} (\gamma_1^{(p)}) \in \mmF_{p+(r-1)} \tilde{L}
		\end{align}  
		and
		\begin{align}
		\label{W20B74}
		\gamma_0^{(p)} (t) \in \mmF_{p+(r-1)} \tilde{L} \hat{\otimes} \mhK[t].
		\end{align}
		Together with $U_\star (a_p) = \psi (a_p) + \mcO (\mmF_{p+r} \tilde{L})$ this clearly implies
		\begin{align}
		\label{W20B77}
		\psi (a_p) -d_{\tilde{L}} (\gamma_1^{(p)})  \in \mmF_{p+r} \tilde{L}
		\end{align}
		to hold.\\
		By assumption we have $a_p \in \mmF_p L$ and $ \gamma_1^{(p)} \in \mmF_p \tilde{L}$. Together with $d_L (a_p) \in \mmF_{p+r} L$ and $\psi (a_p ) - d_{\tilde{L}} (\gamma_1^{(p)}) \in \mmF_{p+r} \tilde{L}$ from Equations \eqref{W20B75} and \eqref{W20B77}, all the requirements of Equation \eqref{W14B14a} are satisfied.\\
		Thus, according to Equation \eqref{W14B14}, we may find:
		\begin{align}
		\label{W20B78}
		&\left\{
		\begin{aligned}
		\exists x &\in \mmF_{p-r+1} L\\
		\exists y &\in \mmF_{p-r+1} \tilde{L}
		\end{aligned}
		\right.\\
		\nonu \mst\\
		\label{W20B79}
		&\left\{
		\begin{aligned}
		d_L (x) &\in \mmF_p L\\
		\psi (x ) - d_{\tilde{L}} (y) &\in \mmF_p \tilde{L}
		\end{aligned}
		\right.\\	
		\nonumber \mand\\
		\label{W20B80}
		&\left\{
		\begin{aligned}
		a_p - d_L (x) &\in \mmF_{p+1}	L\\
		\gamma_1^{(p)} - \psi (x) + d_{\tilde{L}} (y) &\in \mmF_{p+1} \tilde{L}.
		\end{aligned}
		\right.	
		\end{align}
		Next, we define, according to Lemma \ref{Lemma W14B1}, a rectified 1-cell $\xi^{(p+1)}= \xi_0^{(p+1)} (t)+ dt~\xi_1^{(p+1)}$ in $\mmMC(L)$ by means of setting
		\begin{align*}
		\begin{aligned}
		\xi_0^{(p+1)} (0) &\meqd a_p\\
		\xi_1^{(p+1)} &\meqd -x.
		\end{aligned}
		\end{align*}
		Integration of the corresponding IVP up to some arbitrary $t$ results in
		\begin{align*}
		\xi_0^{(p+1)} (t) = \hspace{-5pt} \ubr{a_p}_{ \in \mmF_p L} \hspace{-5pt} + \int_{0}^{t} \hspace{-2pt} dt_1\Big( \hspace{-2pt}\ubr{d_L (-x)}_{\minc{\eqref{W20B79}} \mmF_p L} + \ubr{\sum_{m=1}^{\infty} \frac{1}{m!} \{ \xi_0^{(p+1)} (t_1) , \ldots, \xi_0^{(p+1)} (t_1) , -x  \}_{m+1}}_{=0} \Big),
		\end{align*}
		i.e.
		\begin{align}
		\label{W20B84}
		\xi_0^{(p+1)} (t) \in \mmF_p L \hat{\otimes} \mhK[t].
		\end{align}
		Further, we set $a_{p+1}$ to be the endpoint of this 1-cell and so we find for its degree of filtration
		\begin{align}
		\label{W20B85}
		a_{p+1} \meqd \xi_0^{(p+1)} (1) = \ubr{a_p - d_L (x)}_{\minc{\eqref{W20B80}} \mmF_{p+1} L} \in \mmF_{p+1} L.
		\end{align}
		\underline{Connecting 0 and $U_\star (a_{p+1})$:}\hfill\\
		Next, we construct a 1-cell in $\mmMC(\tilde{L})$ which connects 0 and $U_\star (a_{p+1})$. For doing so, we apply Lemma \ref{Lemma W19A2} with $\xi^{(p+1)}$ in the role of $\xi$ and $\beta = \beta_0 (t) + dt~\beta_1$ given by:
		\begin{align}
		\left\{
		\begin{aligned}
		\beta_0(t) &\meqd \gamma_0^{(p)} (1-t)\\[-3pt]
		\beta_1 &\meqd - \gamma_1^{(p)}.
		\end{aligned}
		\right.	
		\end{align}
		As a result we obtain a 1-cell $\kappa = \kappa_0 (t) + dt~\kappa_1 (t)$ in $\mmMC(\tilde{L})$ connecting 0 and $U_\star (a_{p+1})$ satisfying
		\begin{align}
		\label{W20B86}
		\kappa_1 (t) = \gamma_1^{(p)}- \psi (x)  + \mcO(\mmF_{p+1} \tilde{L} \hat{\otimes} \mhK[t]).
		\end{align}
		\underline{Adjusting the degree of filtration of the connection of $0$ and $U_\star (a_{p+1})$:}\hfill\\
		According to Equation \eqref{W20B86}, $\kappa_1(t)$ can also be rewritten as
		\begin{align}
		\label{W20B87}
		\kappa_1 (t) = s + q(t)
		\end{align}
		for
		\begin{align}
		\label{W20B88}
		s = \gamma_1^{(p)} - \psi (x)
		\end{align}
		(no time dependence) and 
		\begin{align}
		\label{W20B89}
		q(t) \in \mmF_{p+1} \tilde{L} \hat{\otimes} \mhK[t].
		\end{align}
		Since $\kappa (t)$ is a 1-cell in $\mmMC(\tilde{L})$ connecting $0$ and $U_\star (a_{p+1})$ it must be the solution of the corresponding IVP as discussed in Lemma \ref{Lemma W14B1}.\\
		Integration of the IVP up to $t$ and $1$, respectively, leads to
		\begin{align}
		\label{W20B90}
		\kappa_0 (t) = 0 + t d_{\tilde{L}} (s) + \int_{0}^{t} dt_1 \Big(\hspace{-7pt}\ubr{d_{\tilde{L}} (q(t_1))}_{\minc{\eqref{W20B89}} \mmF_{p+1} \tilde{L} \hat{\otimes}\mhK[t_1]} \hspace{-5pt} + \sum_{k=1}^{\infty} \frac{1}{k!} \{ \kappa_0 (t_1), \ldots, \kappa_0 (t_1) , \kappa_1 (t_1)   \}_{k+1} \Big)
		\end{align}
		and
		\begin{align}
		\label{W20B91}
		\ubr{U_\star (a_{p+1})}_{\in \mmF_{p+r} \tilde{L}} = \kappa_0 (1) = 0 + d_{\tilde{L}} (s) + \int_{0}^{1} dt~\Big( \hspace{-7pt}\ubr{d_{\tilde{L}} (q(t))}_{\minc{\eqref{W20B89}} \mmF_{p+1} \tilde{L} \hat{\otimes}\mhK[t]} \hspace{-5pt} + \sum_{k=1}^{\infty} \frac{1}{k!} \{\kappa_0 (t), \ldots, \kappa_0 (t), \kappa_1 (t)  \}_{k+1} \Big).
		\end{align}
		The usual analysis of the degrees of filtration yields
		\begin{align}
		\label{W20B92}
		d_{\tilde{L}} (s) \in \mmF_{p+1} \tilde{L}
		\end{align}
		and
		\begin{align}
		\label{W20B93}
		\kappa_0 (t) \in \mmF_{p+1} \tilde{L} \hat{\otimes} \mhK[t].
		\end{align}
		By the same arguments as in the proof of Lemma \ref{Lemma W14B2} we find
		\begin{align}
		\label{W20B94}
		\mu = \kappa + d_{\tilde{L}}^\kappa (dt~ y)
		\end{align}
		with $y$ as in Equation \eqref{W20B78} to be a 1-cell in $\mmMC(\tilde{L})$ connecting $0$ and $U_\star (a_{p+1})$ as well.\\
		Making use of $dt^2 =0$, $\mu= \mu_0 (t)+ dt~\mu_1 (t)$ amounts to
		\begin{align}
		\label{W20B96}
		\mu = \kappa_0 (t) + \Big( \kappa_1 (t) + d_{\tilde{L}} (y) + \ubr{\{\hspace{-15pt}\ubr{\kappa_0 (t)}_{\minc{\eqref{W20B93}} \mmF_{p+1} \tilde{L} \hat{\otimes} \mhK[t]} \hspace{-15pt},y\}_2 + \text{higher filtration degrees}}_{\in \mmF_{p+1} \tilde{L} \hat{\otimes} \mhK[t]} \Big) dt.
		\end{align}
		By writing out $\kappa_1 (t)$ in the sense of Equation \eqref{W20B86} we obtain
		\begin{align*}
		\mu_1 (t) = \ubr{\gamma_1^{(p)}- \psi (x) + d_{\tilde{L}} (y)}_{\minc{\eqref{W20B80}} \mmF_{p+1} \tilde{L} } + \mcO (\mmF_{p+1} \tilde{L} \hat{\otimes} \mhK[t]) \in \mmF_{p+1} \tilde{L} \hat{\otimes } \mhK[t].
		\end{align*}
		Applying Lemma B2 from \cite{DolgushevLinfty} yields that there exists a rectified 1-cell $\gamma^{(p+1)}= \gamma_0^{(p+1)} (t) + dt~ \gamma_1^{(p+1)}$ in $\mmMC(\tilde{L})$, which connects 0 with $U_\star (a_{p+1})$ and satisfies $\gamma_1^{(p+1)} \in \mmF_{p+1} \tilde{L}$.
	\end{proof}
	
	\section{Lemma for Concatenating 1-Cells}\hfill\\
	The following Lemma is a slightly adjusted version of Proposition 3.3 from \cite{DolgushevLinfty}.
	\begin{lemma}
		\label{Lemma W19A2}	\hfill\\
		Let $L,\tilde{L}$ be two $\mmS L_\infty$ algebras endowed with descending, bounded above and complete filtrations
		\begin{align*}
		\begin{aligned}
		L= \mmF_1 L \supset \mmF_2 L \supset \mmF_3 L \supset \ldots\\
		\tilde{L} = \mmF_1 \tilde{L} \supset \mmF_2 \tilde{L} \supset \mmF_3 \tilde{L} \supset \ldots
		\end{aligned}
		\end{align*}
		compatible with the $\mmS L_\infty$ structures\footnote{In the setting of Theorem \ref{Theorem ZA1} the lemma also holds for the filtration of $L$ starting at $\mmF_0 L$ instead.}.\\
		Let $U: L \rightarrow \tilde{L}$ be an $\infty$- morphism of $\mmS L_\infty$ algebras compatible with the filtrations and denote its linear part by $\psi$.\\
		Let $m_i$ for $i=0,1,2$ be Maurer-Cartan elements in $L$ and denote their images under $U_\star$, which are also Maurer-Cartan elements in $\tilde{L}$, by $\tilde{m}_i$. Moreover, let $\tilde{m}_1$ satisfy $\tilde{m}_1 \in \mmF_p \tilde{L}$.\\
		Let $\beta = \beta_0 (t ) + dt~\beta_1$ be a rectified 1-cell in $\mmMC(\tilde{L})$ which connects the two Maurer-Cartan elements $\tilde{m}_1$ and $\tilde{m}_0$ and has $\beta_1 \in \mmF_p \tilde{L}$.\\
		Let $\xi = \xi_0 (t) + dt~ \xi_1$ be a rectified 1-cell in $\mmMC(L)$ which connects the two Maurer-Cartan elements $m_1$ and $m_2$.\\
		In addition, let us assume $d_L (\xi_1) \in \mmF_p L$ and $\xi_0 (t) \in \mmF_p L \hat{\otimes } \mhK[t]$.\\
		Then there exists a 1-cell
		\begin{align}
		\label{W19A1}
		\kappa = \kappa_0 (t) + dt ~\kappa_1 (t)
		\end{align}
		in $\mmMC(\tilde{L})$, connecting $\tilde{m}_0$ with $\tilde{m}_2$ and satisfying
		\begin{align}
		\label{W19A2}
		\kappa_1 (t) = - \beta_1 + \psi (\xi_1) + \mcO(\mmF_{p+1}  \tilde{L} \hat{\otimes } \mhK[t]).
		\end{align}
	\end{lemma}
	\begin{proof}[\textbf{Proof of Lemma \ref{Lemma W19A2}}]
		\hfill\\                           
		We set
		\begin{align}
		\label{W14B131}
		\chi = \chi_0 (t) + dt~ \chi_1 (t) \meqd U_\star^{(1)} (\xi_0 (t) + dt~\xi_1), 
		\end{align}
		i.e. $\chi$ is a 1-cell in $\mmMC(\tilde{L})$ which connects $\tilde{m}_1$ and $\tilde{m}_2$.\\
		As a consequence, we have the horn
		\[
		\begin{tikzcd}
		& \tilde{m}_1 \arrow[ld, "\chi = \chi_0 (t_2) + dt_2~\chi_1 (t_2)"'] \arrow[rd, "\beta = \beta_0 (t_0) + dt_0~\beta_1"] &   \\
		\tilde{m}_2 &                                                                                                                                                      & \tilde{m}_0
		\end{tikzcd},
		\]
		which is the same as in the starting point of Proposition 3.3 from \cite{DolgushevLinfty}.\\
		For our next considerations, we will heavily rely on the notation of Appendix A from \cite{DolgushevLinfty}, without elaborating further on the specifics of it.\\
		According to Lemma A1 from \cite{DolgushevLinfty}, we know that
		\begin{align}
		\label{W14B133}
		\left\{
		\begin{aligned}
		\mMC(\tilde{L} \hat{\otimes} \Omega_n) &\stackrel{\cong}{\rightarrow} \mMC(\tilde{L}) \times \mStub_n^i (\tilde{L})\\
		a &\mapsto \big( \epsilon_i^{n} (a), (d_{\tilde{L}} + d_{\Omega_n}) \circ h_n^i (a) \big)
		\end{aligned}
		\right.	
		\end{align}
		is a bijection.\\
		For the 1-cell $\beta$ the corresponding (in the sense of Equation \eqref{W14B133}) $\mMC(\tilde{L})$ and $\mStub_1^1 (\tilde{L})$ elements are found to be
		\begin{align*}
		\epsilon_1^1 (\beta) = \tilde{m}_1
		\end{align*}
		and
		\begin{align*}
		(d_{\tilde{L}}+ d_{\Omega_1}) h_1^1 (\beta) = (d_{\tilde{L}}+ d_{\Omega_1}) (t_0 \beta_1).
		\end{align*}
		For $\chi$,
		\begin{align*}
		\epsilon_1^1 (\chi) = \tilde{m}_1
		\end{align*}
		is immediate.\\
		For the $\mStub_1^1 (\tilde{L})$ element corresponding to $\chi$, explicit calculation is quite difficult due to the time dependence of its $dt$ coefficient.\\
		Hence, most of the time we simply write
		\begin{align*}
		(d_{\tilde{L}}+ d_{\Omega_1}) h_1^1 (\chi) = (d_{\tilde{L}}+ d_{\Omega_1}) \lambda (t_2).
		\end{align*}
		However, having a second look unravels
		\begin{align}
		\label{W14B137a}
		\lambda (t_2) = h_1^1 (\chi) = t_2 \int_{0}^{1} du~ \chi_1 (u t_2),
		\end{align}
		i.e. there is a $t_2$ pre-factor.\\
		We define
		\begin{align}
		\label{W14B138}
		\tilde{\nu} \meqd (d_{\tilde{L}} + d_{\Omega_2}) ( t_0 \beta_1 + \lambda (t_2)),
		\end{align}
		which is exactly the sum of the two stub elements associated to the 1-cells on the edges of the horn.\\
		A short calculation confirms $\tilde{\nu}$ to lie in $\mStub_2^1 (\tilde{L})$.\\
		Let $\tilde{\eta}$ be the 2-cell corresponding to $(\tilde{m}_1, \tilde{\nu})$ (in the sense of Equation \eqref{W14B133}).\\
		By recalling the explicit form of $h_1^1(\beta)$ and the $t_2$ prefactor of $\lambda(t_2)$, it is clear that both
		\begin{align*}
		\tilde{\nu} \vert_{t_2=0} = (d_{\tilde{L}} + d_{\Omega}) (t_0 \beta_1)
		\end{align*}
		and
		\begin{align*}
		\tilde{\nu} \vert_{t_0=0} = (d_{\tilde{L}} + d_{\Omega}) (\lambda (t_2))
		\end{align*}
		hold.\\
		But, according to Equation \eqref{W14B133}, this directly implies
		\begin{align*}
		\tilde{\eta} \vert_{t_2=0} = \beta
		\end{align*}
		and
		\begin{align*}
		\tilde{\eta} \vert_{t_0=0} = \chi.
		\end{align*}
		So, we have proved $\tilde{\eta}$ to be the filler of the horn.\\
		Our next task is to analyse the filtration degree at the bottom edge $(t_1 =0)$.\\
		According to Appendix A of \cite{DolgushevLinfty}, the 2-cell $\tilde{\eta}$ corresponding to $(\tilde{m}_1 , \tilde{\nu})$ is found to be the limiting value of the sequence $\{\eta^{(k)}\}_{k \geq 0}$ defined by
		\begin{align}
		\label{W14B143}
		\left\{
		\begin{aligned}
		\tilde{\eta}^{(0)} &\meqd \tilde{m}_1 + \tilde{\nu}\\
		\tilde{\eta}^{(k+1)} &\meqd \tilde{\eta}^{(0)} - \sum_{m=2}^{\infty} \frac{1}{m!} h_2^1 \{ \tilde{\eta}^{(k)}, \ldots, \tilde{\eta}^{(k)} \}_m.
		\end{aligned}
		\right.	
		\end{align}
		Thus in particular,
		\begin{align}
		\label{W14B144}
		\tilde{\eta} = \tilde{m}_1 + \tilde{\nu} -\sum_{m=2}^{\infty} \frac{1}{m!} h_2^1 \{\tilde{\eta}, \ldots, \tilde{\eta}\}_{m}
		\end{align}
		holds.\\
		Let us now prove that
		\begin{align}
		\label{W14B145}
		\tilde{\eta} - \tilde{m}_1 - \tilde{\nu} \in \mmF_{p+1} \tilde{L} \hat{\otimes } \Omega_2.
		\end{align}
		By Equation \eqref{W14B144}, it is immediate that
		\begin{align}
		\label{W14B146}
		\tilde{\eta} - \tilde{m}_1 - \tilde{\nu} = - \sum_{m=2}^{\infty} \frac{1}{m!} h_2^1 \{\tilde{\eta}, \ldots, \tilde{\eta}\}_m.
		\end{align}
		From the definition of $h_2^1$, only terms involving a $dt_i$ survive.\\
		So it boils down to an inspection of the filtration degrees of the possible combinations of terms emerging from $\tilde{m}_1$ and $\tilde{\nu}$, which have at least one $dt_i$ term.\\
		A short investigation of $\tilde{\nu}$ shows that the only combination of terms involving a $dt_i$ and not trivially having filtration degree $\mmF_{p+1} \tilde{L}$ is the one of $d_{\tilde{L}} (\lambda(t_2))$ with $dt_2 \frac{\partial \lambda (t_2)}{\partial t_2}$ (among others we used the fact that $\beta_1 \in \mmF_p \tilde{L}$ by assumption).\\
		We continue by discussing why $d_{\tilde{L}} ( \lambda (t_2)) \in \mmF_p \tilde{L} \hat{\otimes } \mhK[t_2]$ holds.\\
		In Equation \eqref{W14B137a} we have defined $\lambda(t_2)$ to be the integral of $\chi_1 (t)$. As integration and $d_{\tilde{L}}$ can be swapped, it is clear that $d_{\tilde{L}} (\chi_1 (t_2)) \in \mmF_p \tilde{L} \hat{\otimes } \mhK[t_2]$ would suffice.\\
		In Equation \eqref{W14B131} we have defined $\chi_1 (t)$ to be the $dt$ coefficient of $U_\star^{(1)} (\xi_0 (t) +dt~ \xi_1 )$ and so $\chi_1$ is of the form
		\begin{align}
		\label{W14B147}
		\chi_1(t) = \sum_{m=1}^{\infty} \frac{1}{(m-1)!} U^\prime (\underbrace{ \xi_0 (t), \ldots, \xi_0 (t),\xi_1}_{\text{m elements}}).
		\end{align}
		From Equation \eqref{W14B147} and Proposition 10.2.7. from \cite{Loday}, it is immediate that $d_{\tilde{L}} (\chi_1 (t_2))$ only consists of terms, where $U^\prime$ has in its argument a
		\begin{align}
		\label{W14B148}
		d_L (\xi_1)
		\end{align}
		term, and/or a
		\begin{align}
		\label{W14B149}
		\xi_0 (t)
		\end{align}
		term or a $d_L$ derivation thereof.\\
		But by assumption both carry filtration degree $\mmF_p L$ and $\mmF_p L \hat{\otimes} \mhK[t]$, respectively. Consequently, this proves $d_{\tilde{L}} (\lambda (t_2)) \in \mmF_p \tilde{L} \hat{\otimes } \mhK[t_2]$ and thus the r.h.s of Equation \eqref{W14B146} in any case lies in $\mmF_{p+1} \tilde{L} \hat{\otimes } \Omega_2$, i.e. Equation \eqref{W14B145} holds.\\
		Due to Equation \eqref{W14B147} and $\xi_0 (t) \in \mmF_p L \hat{\otimes } \mhK[t]$, one can directly compute
		\begin{align*}
		\chi_1 (t_2) - \psi (\xi_1) \in \mmF_{p+1} \tilde{L} \hat{\otimes } \mhK[t_2].
		\end{align*}
		Applying $h_1^1$ on both sides (formally, we should consider the whole expression carrying a $dt_2$) leads to
		\begin{align}
		\label{W14B151}
		\lambda (t_2) - t_2 \psi (\xi_1) \in \mmF_{p+1} \tilde{L} \hat{\otimes } \mhK[t_2].
		\end{align}
		This allows us to compute $\tilde{\nu}$ evaluated at the bottom boundary:
		\begin{align}
		\label{W14B152}
		\left\{
		\begin{aligned}
		&\tilde{\nu} \vert_{t_1=0} \meqc{\eqref{W14B138}}  \left(t_0 d_{\tilde{L}} (\beta_1) + dt_0~ \beta_1 + (d_{\tilde{L}} + d_{\Omega_2}) (\lambda (t_2)) \right) \vert_{t_1=0}\\
		&\hspace{-8pt}\meqc{\eqref{W14B151}} \left( t_0 d_{\tilde{L}} (\beta_1) + dt_0~\beta_1 + (d_{\tilde{L}} + d_{\Omega_2}) (t_2 \psi (\xi_1) + \mcO(\mmF_{p+1} \tilde{L} \hat{\otimes} \mhK[t_2])) \right)\hspace{-2pt} \vert_{t_1=0} \\
		&=(1-t_2) d_{\tilde{L}} (\beta_1) + t_2 d_{\tilde{L}} (\psi (\xi_1)) + dt_2 \left(\psi (\xi_1)  - \beta_1 \right) + \mcO(\mmF_{p+1} \tilde{L} \hat{\otimes} \Omega_1).
		\end{aligned}
		\right.	
		\end{align}
		The bottom edge of the 2-cell $\tilde{\eta}$ defines a 1-cell by means of
		\begin{align}
		\label{W14B153}
		\kappa = \kappa_0 (t_2) +dt_2~ \kappa_1 (t_2) \meqd \tilde{\eta} \vert_{t_1=0},
		\end{align}
		i.e. we arrive at the situation
		\[
		\begin{tikzcd}
		& \tilde{m}_1 \arrow[rd, "\beta = \beta_0 (t_0) + dt_0~\beta_1"] \arrow[ld, "\chi =\chi_0 (t_2) + dt_2~\chi_1 (t_2)"'] &                                                                                  \\
		\tilde{m}_2 &                                                                                                                      & \tilde{m}_0 \arrow[ll, "\kappa = \kappa_0 (t_2) + dt_2~\kappa_1 (t_2)"]
		\end{tikzcd}.
		\]
		On one hand, we know from Equations \eqref{W14B145} and \eqref{W14B153} that
		\begin{align}
		\label{W14B154}
		\kappa_0 (t_2) + dt_2 ~\kappa_1 (t_2) \meqc{\eqref{W14B153}} \tilde{\eta} \vert_{t_1=0} \meqc{\eqref{W14B145}} \tilde{m}_1 + \tilde{\nu} \vert_{t_1=0} +\mcO (\mmF_{p+1} \tilde{L} \hat{\otimes } \Omega_1).
		\end{align}
		On the other hand, we can make use of Equation \eqref{W14B152} for computing the right-hand side of Equation \eqref{W14B154}, which leads to
		\begin{align*}
		\begin{aligned}
		&\kappa_0 (t_2) + dt_2~\kappa_1 (t_2)\\
		&=\tilde{m}_1 + (1-t_2) d_{\tilde{L}} (\beta_1) + t_2 d_{\tilde{L}} (\psi (\xi_1)) + dt_2 \left(\psi (\xi_1)  - \beta_1 \right) + \mcO(\mmF_{p+1} \tilde{L} \hat{\otimes} \Omega_1).
		\end{aligned}
		\end{align*}
		Therefore, we may deduce
		\begin{align}
		\label{W14B156}
		\kappa_1 (t)= -\beta_1 +\psi (\xi_1) + \mcO(\mmF_{p+1} \tilde{L} \hat{\otimes }\mhK[t]).
		\end{align}\\
		So, we have found a 1-cell $\kappa = \kappa_0 (t) + dt~ \kappa_1 (t)$ in $\mmMC(\tilde{L})$, connecting $\tilde{m}_0$ with $\tilde{m}_2$ and satisfying Equation \eqref{W19A2}.\\
	\end{proof}
	\bibliographystyle{plain}
	\bibliography{GMTv2}
\end{document}